\documentclass[11pt]{amsart}

\pdfoutput=1

\usepackage{amsmath} % I added this package to use \tag{}.
\usepackage[foot]{amsaddr}

\usepackage{cancel}
\usepackage{esint,amssymb} 
\usepackage{graphicx}
\usepackage{MnSymbol}
\usepackage{mathtools} 
\usepackage[colorlinks=true, pdfstartview=FitV, linkcolor=blue, citecolor=blue, urlcolor=blue,pagebackref=false]{hyperref}
\usepackage{orcidlink}
\usepackage{microtype}

\usepackage{bm}
\usepackage{dsfont}
\usepackage{mathrsfs}
\usepackage{enumitem}

\definecolor{darkgreen}{rgb}{0,0.5,0}
\definecolor{darkblue}{rgb}{0,0,0.7}
\definecolor{darkred}{rgb}{0.9,0.1,0.1}

\newtheorem{proposition}{Proposition}
\newtheorem{theorem}[proposition]{Theorem}
\newtheorem{lemma}[proposition]{Lemma}
\newtheorem{corollary}[proposition]{Corollary}

\newtheorem{conjecture}[proposition]{Conjecture}
\theoremstyle{remark}
\newtheorem{remark}[proposition]{Remark}

\theoremstyle{definition}
\newtheorem{definition}[proposition]{Definition}

\numberwithin{equation}{section}
\numberwithin{proposition}{section}
\numberwithin{figure}{section}
\numberwithin{table}{section}
\newcommand{\N}{\mathbb{N}}

\newcommand{\R}{\mathbb{R}}

\newcommand{\E}{\mathbb{E}}
\renewcommand{\P}{\mathbb{P}}

\renewcommand{\le}{\leqslant}
\renewcommand{\ge}{\geqslant}
\renewcommand{\leq}{\leqslant}
\renewcommand{\geq}{\geqslant}

\renewcommand{\subset}{\subseteq}
\renewcommand{\bar}{\overline}
\renewcommand{\tilde}{\widetilde}

\newcommand{\Ll}{\left}
\newcommand{\Rr}{\right}
\renewcommand{\d}{\mathrm{d}}
\newcommand{\dr}{\partial}

\newcommand{\tsp}{\intercal}

\newcommand{\mcl}{\mathcal}
\newcommand{\msf}{\mathsf}

\newcommand{\de}{\delta}

\newcommand{\si}{\sigma}

\newcommand{\upa}{\uparrow}

\newenvironment{e}{\begin{equation}}{\end{equation}\ignorespacesafterend}
\newenvironment{e*}{\begin{equation*}}{\end{equation*}\ignorespacesafterend}

\begin{document}

\author{Victor Issa\,\orcidlink{0009-0009-1304-046X}}
\address[Victor Issa]{Department of Mathematics, ENS de Lyon, Lyon, France}
\email{\href{mailto:victor.issa@ens-lyon.fr}{victor.issa@ens-lyon.fr}}

% \author{Jean-Christophe Mourrat}
% \address[Jean-Christophe Mourrat]{Department of Mathematics, ENS Lyon and CNRS, Lyon, France}
% \email{\href{mailto:jean-christophe.mourrat@ens-lyon.fr}{jean-christophe.mourrat@ens-lyon.fr}}

\subjclass[2010]{}

\title{A Hopf-like formula for mean-field spin glass models}

\begin{abstract}
We study mean-field spin glass models with general vector spins and convex covariance function. For those models, it is known that the limit of the free energy can be written as the supremum of a functional, this is the celebrated Parisi formula. 

In this paper, we observe that the Parisi functional extends into a concave and Lipschitz functional on the set of signed measures. We use this fact and Fenchel-Moreau duality to derive an un-inverted version of the Parisi formula. Namely, we show that the limit of the free energy can be written as the infimum of a functional related to the Parisi functional.  

This un-inverted formula can be interpreted as a Hopf-like formula for some Hamilton-Jacobi equation in Wasserstein space.

\bigskip


    \medskip

\end{abstract}

\maketitle

\newpage
\thispagestyle{empty}
{
  \hypersetup{linkcolor=black}
  \tableofcontents
}

%
%
%
%%%%%%%%%%%%%%%%%%%%%%%%%%%%
%%%%%%%%%%%%%%%%%%%%%%%%%%%%
%
%
%

\newpage 
\pagenumbering{arabic}
\section{Introduction}

\subsection{Preamble}
Let $D \ge 1$ be an integer, and let $(H_N(\sigma))_{\sigma \in (\R^N)^D}$ be a centered Gaussian field such that, for every $\sigma = (\sigma_1,\ldots, \sigma_D)$ and $\tau = (\tau_1, \ldots, \tau_D) \in (\R^N)^D$, we have
\begin{equation} \label{e.def.xi}
    \E \Ll[ H_N(\sigma) H_N(\tau) \Rr] = N \xi \Ll( \frac{\si \tau^\tsp}{N} \Rr) ,
\end{equation}
where $\xi \in C^\infty(\R^{D\times D}; \R)$ is some fixed smooth function, and where $\sigma \tau^\tsp$ denotes the matrix of scalar products 
\begin{equation} \label{e.def.sigma.tau}
    \sigma \tau^\tsp = (\sigma_d \cdot \tau_{d'})_{1 \le d,d' \le D}.
\end{equation}
We often identify $(\R^N)^D$ with the space $\R^{D\times N}$ of $D$-by-$N$ matrices, which makes the notation in \eqref{e.def.sigma.tau} natural. We assume that $\xi$ has a convergent power-series expansion. We also give ourselves a reference measure $P_1$ on $\R^D$ of compact support, and we set $P_N = (P_1)^{\otimes N}$, which we think of as a probability measure on $\R^{D\times N}$. In other words, a sample $\sigma = (\sigma_{d,i})_{1 \le d \le D, 1 \le i \le N}$ from $P_N$ is such that the columns $(\sigma_{\cdot,i})_{1 \le i \le N}$ are independent with law $P_1$. We are interested in the study of the large-$N$ behavior of the free energy
\begin{equation} \label{e.def.barFN}
    \bar F_N(t,\de_0) = -\frac 1 N \E \log \int \exp\Ll(\sqrt{2t} H_N(\sigma) - t N \xi \Ll( \frac{\si \si^\tsp}{N} \Rr) \Rr) \, \d P_N(\si), 
\end{equation}
where $t \ge 0$. The term $\xi\Ll( \frac{\si \si^\tsp}{N} \Rr)$ in \eqref{e.def.barFN} is introduced as a convenience to simplify the expression of the limit; it is constant in classical cases of interest, such as when the coordinates of $\sigma$ takes values in $\{-1,1\}$ and $\xi$ depends only on the diagonal entries of its argument. In general, the second argument of $\bar F_N$ can be any monotone probability measure on the space $S^D_+$ of positive semi-definite matrices, subject to a mild integrability requirement; the expression in \eqref{e.def.barFN} is with this argument chosen to be the Dirac mass at the origin. To explain what this space is, let us say that a path $\msf q : [0,1) \to S^D_+$ is nondecreasing if for every $u \le v \in [0,1)$, we have $\msf q(v) - \msf q(u) \in S^D_+$. A probability measure on $S^D_+$ is said to be monotone if it is the image of the Lebesgue measure on $[0,1]$ through a nondecreasing path from $[0,1)$ to $S^D_+$. We write $\mcl P^\upa(S^D_+)$ to denote the set of monotone measures on $S^D_+$, which is a subset of the set $\mcl P(S^D_+)$ of probability measures on $S^D_+$. For every $p \in [1,+\infty]$, we also write $\mcl P_p(S^D_+)$ to denote the subspace of $\mcl P(S^D_+)$ of measures with finite $p$-th moment, with the understanding that $\mcl P_\infty(S^D_+)$ is the subset of probability measures with compact support; we also write $\mcl P^\upa_p(S^D_+) = \mcl P^\upa(S^D_+) \cap \mcl P_p(S^D_+)$. We postpone the precise definition of $\bar F_N(t,\mu)$ for arbitrary $\mu \in \mcl P^\uparrow_1(S^D_+)$ to \eqref{e.def enriched fe}.  In short, this quantity is obtained by adding an energy term in the exponential on the right side of \eqref{e.def.barFN} to encode the interaction of $\sigma$ with an external magnetic field, and this external magnetic field has an ultrametric structure whose characteristics are encoded by the measure $\mu$. 

One can check \cite[Proposition~3.1]{mourrat2019parisi} that $\bar F_N(0,\cdot)$ does not depend on $N$; for every $\mu \in \mcl P^\upa_1(S^D_+)$, we write
\begin{equation*}  %\label{e.}
    \psi(\mu) = \bar F_1(0,\mu) = \bar F_N(0,\mu). 
\end{equation*}
This follows from the fact $P_N = P_1^{\otimes N}$ and that at $t = 0$ the $N$-body Hamiltonian has the same law as $N$ copies of the $1$-body Hamiltonian. When instead $P_N$ is the uniform measure on the sphere of radius $\sqrt{N}$ centered at $0$ in $(\mathbb{R}^D)^N$, $\bar F_N(0,\cdot)$ depends on $N$ but converges to a smooth function of $\mu$ as $N \to +\infty$ \cite[Proposition~3.1]{mourrat2019parisi}. In what follows, we focus on models with $P_N = P_1^{\otimes N}$.

When $\xi$ is convex on $S^D_+$, the limiting value of $\bar F_N(t,\delta_0)$ is known, this is the celebrated Parisi formula. The Parisi formula was first conjectured in \cite{parisi1979infinite} using a sophisticated non-rigorous argument now referred to as the replica method. The convergence of the free energy as $N \to +\infty$ was rigorously established in \cite{guerra2002} in the case of the so-called Sherrington-Kirkpatrick model which corresponds to $D = 1$, $\xi(x) = x^2$ and $P_1 = \text{Unif}(\{-1,1\})$. The Parisi formula for the Sherrington-Kirkpatrick model was then proven in \cite{gue03,Tpaper}. This was extended to the case $D = 1$, $P_1 = \text{Unif}(\{-1,1\})$ and $\xi(x) = \sum_{p \geq 1} a_p x^p$ with $a_p \geq 0$ in \cite{pan}. Some models with $D > 1$ such as multispecies models, the Potts model, and a general class of models with vector spins were treated in \cite{pan.multi,pan.potts,pan.vec}, under the assumption that $\xi$ is convex on $\R^{D \times D}$. Finally, the case $D > 1$ was treated in general in \cite{chenmourrat2023cavity} assuming only that $\xi$ is convex on the set of positive semi-definite matrices. The following version of the Parisi formula is \cite[Corollary~8.2]{chenmourrat2023cavity}.

\begin{theorem} [\cite{chenmourrat2023cavity}] \label{t.parisi}
    If $\xi$ is convex on $S^D_+$, then we have for every $t > 0$
    \begin{e} \label{e.parisi}
        \lim_{N \to +\infty} \bar F_N(t,\delta_0) = \sup_{\mu \in \mcl P^\upa_\infty(S^D_+)} \left\{ \psi(\mu)  - t \int \xi^* \left( \frac{x}{t}\right) \d \mu(x) \right\}.    
    \end{e}
\end{theorem}

Here, $\xi^*$ denotes the convex dual of $\xi$ with respect to the cone $S^D_+$, it is the function $\R^{D \times D} \to \R \cup\{+\infty\}$ defined by
\begin{equation} \label{e.def.xi.star}
    \xi^*(a) = \sup_{b \in S^D_+} \Ll\{a \cdot b - \xi(b)\Rr\}.
\end{equation}

\subsection{Main results}

In the classical version of the Parisi formula \eqref{e.parisi}, the limit free energy is written as the supremum of a functional. In this paper, we manipulate the right-hand side of \eqref{e.parisi} to show that there exists a functional related to $\psi$ whose infimum is equal to the limit of the free energy. Unless stated otherwise, we always work under the following assumptions, they are used to make sure that Theorem~\ref{t.parisi} holds.

\begin{enumerate}[start=1,label={\rm (H\arabic*)}]
    \item \label{i.assume_1_product_proba}
    For every $N \geq 1$, $P_N = P_1^{\otimes N}$.
    \item \label{i.assume_2_convexity}
    The function $\xi$ is convex on $S^D_+$.
    \item \label{i.assume_3_power_series}
    The function $\xi$ admits an absolutely convergent power series.
\end{enumerate}
Here and throughout, we will use $\int h \d\mu$ as a shorthand for $\int_{S^D_+} h(x) \d\mu(x)$. When $D = 1$ we have that $\mcl P^\upa(S^D_+) = \mcl P(\R_+)$ is a convex set, and we have the result of \cite{auffinger2015parisi} on the convexity of the Parisi functional, which is essentially the mapping $-\psi$, up to a linear term and a change of variables (see also \cite{mourrat2019parisi}). This motivates the introduction of $\psi_*$ the concave dual of $\psi$. The function $\psi_*$ is defined for every Lipschitz function $\chi : \R_+ \to \R$ by 
\begin{e*}
    \psi_*(\chi) = \inf_{\mu \in \mcl P_1(\R_+)} \left\{ \int \chi \d\mu - \psi(\mu) \right\}.
\end{e*}
Given a Lipschitz function $\chi : \R_+ \to \R$, we also define for every $x \in \R_+$,
\begin{e} \label{e.s_t chi}
    S_t \chi(x) = \sup_{y \in \R_+} \left\{ \chi(x+y) - t \xi^* \left( \frac{y}{t} \right) \right\}.
\end{e}
We recall that $\bar F_N(t,\mu)$ for $\mu \neq \delta_0$ is defined in \eqref{e.def enriched fe}. %\viccomment{Apparently there are some results about the convexity of $\psi$ in \cite{HJbook}, did not find them.}

\begin{theorem} \label{t.main scalar}
   Assume \ref{i.assume_1_product_proba}, \ref{i.assume_2_convexity} and \ref{i.assume_3_power_series}, also assume that $D = 1$, then for every $t \geq 0$ and $\mu \in \mcl P_1(\R_+)$, we have
    \begin{e} \label{e.main scalar}
        \lim_{N \to +\infty} \bar F_N(t,\mu) = \inf_{\chi} \left\{ \int S_t \chi \d\mu - \psi_*(\chi) \right\},
    \end{e}
    where the infimum is taken over the set of $1$-Lipschitz, convex and nondecreasing functions $\chi : \R_+ \to \R$. 
\end{theorem}

We add that when $\chi$ is convex, according to \cite[Proposition~6.3]{chen2023viscosity}, for every $x \in \R_+$, the quantity $S_t \chi(x)$ can also be represented in the following way,
\begin{e} \label{e.chi hopf}
    S_t \chi(x) = \sup_{y \in \R_+} \left\{ xy - \chi^*(y) + t \xi(y) \right\},
\end{e}
where $\chi^*(y) = \sup_{x \in \R_+} \left\{ xy - \chi(x) \right\}$ is the convex dual of $\chi$.

We equip $S^D$, the set of $D \times D$ symmetric matrices, with the Frobenius inner product,
\begin{e*}
    x \cdot y = \sum_{d,d' = 1}^D x_{dd'}y_{dd'}.
\end{e*}
We let $|\cdot|$ denote the associated norm and  $B(0,1)$ denotes the centered unit ball in $S^D$ with respect to $|\cdot|$ and $\mcl K_\xi = \nabla \xi(S^D_+ \cap B(0,1))$. When $D > 1$, the set $\mcl P^\upa(S^D_+)$ is not convex. Due to this technical difficulty, we were unable to obtain the exact analog of \eqref{e.main scalar} in this case. To circumvent this difficulty, we build a Lipschitz extension of $\psi$ defined on the set of signed measures on $\mcl K_\xi$. This yields a formula closely related to \eqref{e.main scalar} in which $\psi_*$ and $S_t \chi$ are replaced by the following
\begin{e}
    \psi_*^\xi(\chi) = \inf_{ \mu \in \mcl P^\upa(S^D_+) \cap \mcl P(\mcl K_\xi)} \left\{ \int \chi d\mu - \psi(\mu) \right\},
\end{e}
\begin{e}
    \tilde S_t \chi(x) = \sup_{y \in S^D_+ \cap B(0,1)} \left\{ \chi(x + t\nabla \xi(y)) - t \xi^* \left(\nabla \xi(y) \right) \right\}.
\end{e}
\begin{theorem} \label{t.main}
    Assume \ref{i.assume_1_product_proba}, \ref{i.assume_2_convexity} and \ref{i.assume_3_power_series}, then for every $t \geq 0$ we have
    \begin{e} \label{e.main}
        \lim_{N \to +\infty} \bar F_N(t,\delta_0) = \inf_{\chi} \left\{ \tilde S_t \chi(0) - \psi^\xi_*(\chi) \right\},
    \end{e}
    where the infimum is taken over the set of $1$-Lipschitz functions $\chi : S^D_+ \to \R$.
\end{theorem}
When mean-field spin glass models were first studied, it seemed natural to believe that the limit free energy could be written as the infimum of a functional. In this sense \eqref{e.parisi} which was first put forward in \cite{parisi1979infinite}, can be referred to as an \emph{inverted} variational formula. Below, we will therefore refer to formulas such as \eqref{e.main scalar} and \eqref{e.main} as \emph{un-inverted} variational formulas. We point out that a different un-inverted formula has already been obtained in \cite[Theorem~1]{uninverting}, where it is shown that the limit free energy can be written as an infimum over a set of martingales.

We finish this introduction by discussing some possible interpretations and generalizations of Theorems \ref{t.main scalar} and \ref{t.main}. We rely on the fact that, as pointed out in \cite{mourrat2019parisi}, the limit free energy is related to the viscosity solution of some Hamilton-Jacobi equation on $\mcl P^\upa(S^D_+)$ (see Theorem~\ref{t.limit.free.energy} below).

One can show that any concave upper semicontinuous function on $\R^d$ coincides with the infimum of the family of affine functions that upper bound it, this is Fenchel-Moreau duality \cite[Theorem~2.3.3]{convexanalysisvectorspace}. The function $u(t,x) = a \cdot x + b + tH(a)$ can be interpreted as the unique solution (in the viscosity sense) of the Hamilton-Jacobi equation
\begin{e*}
    \begin{cases}
        \partial_t u - H(\nabla u) = 0 \text{ on } (0,+\infty) \times \R^d \\
        u(0,x) = a \cdot x + b.
    \end{cases}
\end{e*}
Given a Hamilton-Jacobi equation on $\R^d$ with concave and upper semicontinuous initial condition $\varphi$, one can represent its unique solution (in the viscosity sense) as the infimum of the family of solutions with affine initial conditions that upper-bound $\varphi$, this is the Hopf representation \cite[Theorem~3.13]{HJbook}. Note that for the Hopf representation to hold, it is \emph{not} required to assume that $p \mapsto H(p)$ is convex.

For Hamilton-Jacobi equations on $\mcl P^\upa(S^D_+)$, there is an important subtlety and more precision is needed to define concave and affine functions. There are two natural notions of geodesics on $\mcl P^\upa(S^D_+)$. The first are the geodesics inherited from the inclusion of $\mcl P^\upa(S^D_+)$ into the space of signed measures on $S^D_+$, those are the straight lines $\lambda \mapsto \lambda \mu + (1-\lambda)\mu'$. The second are the transport geodesics inherited from the inclusion of $\mcl P^\upa(S^D_+)$ into $L^2([0,1),S^D_+)$ via the space of nondecreasing paths $\msf q : [0,1) \to S^D_+$, those are the images of the straight lines  $\lambda \mapsto \lambda \msf q + (1-\lambda)\msf q'$ in $\mcl P^\upa(S^D_+)$. This second kind of geodesics are \emph{not} straight lines in $\mcl P^\upa(S^D_+)$ and can be seen as the geodesics associated to the optimal transport geometry in $\mcl P^\upa(S^D_+)$ . 

For the Hamilton-Jacobi equation on $\mcl P^\upa(S^D_+)$ arising in the context of spin glasses (see \eqref{e.hj} below), both of those geometries play a role. The transport derivative $\partial_\mu$ in \eqref{e.hj} is defined using the transport geometry and tracks the infinitesimal slope of a function along the transport geodesics. On the other hand, the initial condition $\psi$ in \eqref{e.hj} is concave along straight lines $\lambda \mapsto \lambda \mu + (1-\lambda)\mu'$.

In \cite[Theorem~1.1~(3)]{chen2023viscosity}, it was shown that if the initial condition in \eqref{e.hj} is concave along transport geodesics, then, perhaps unsurprisingly, the Hopf representation holds. But, it is known that this concavity assumption is not satisfied by the initial condition of \eqref{e.hj} in general \cite[Section~6]{mourrat2020nonconvex}.

At first glance, the two competing geometries in the formulation of \eqref{e.hj} prevents the existence of a Hopf type representation for the solution. Indeed, the initial condition is concave along straight lines, so the relevant affine functions to consider for the Fenchel-Moreau duality are of the form $\mu \mapsto \int \chi(x) \d\mu(x)$. But, for a Hamilton-Jacobi equation on $\mcl P^\upa(S^D_+)$ formulated with the transport derivative $\partial_\mu$ and with a fully general nonlinearity of the form $H(\partial_\mu f)$, there is no reason for the solution started with the initial condition $\mu \mapsto \int \chi(x) \d\mu(x)$ to be affine along straight lines. In the context of \eqref{e.hj}, there are additional structures to exploit in the nonlinear term $\int \xi(\partial_\mu f)\d\mu$ and this apparent incompatibility between the two geometries is resolved. We show in Theorem~\ref{t.viscosity with linear initial condition} that the solution of \eqref{e.hj} with initial condition $\mu \mapsto \int \chi(x) \d\mu(x)$ is 
\begin{e*}
    (t,\mu) \mapsto \int S_t \chi(x) \d\mu(x),
\end{e*}
where $S_t \chi$ is defined in \eqref{e.s_t chi}. This proves that \eqref{e.hj} does preserve affine functions (at least when $D = 1$). 

Theorem~\ref{t.main scalar} is in fact the statement that when $D =1$, the solution of \eqref{e.hj} is the infimum of the family of solutions with affine initial conditions that upper-bound $\psi$ and \eqref{e.main scalar} can therefore be interpreted as a Hopf-like formula for the solution of \eqref{e.hj}.

When $\xi$ is not assumed to be convex on $S^D_+$, the Parisi formula breaks down and the value of the limit free energy is not known. In \cite{mourrat2019parisi} it was conjectured that, in analogy with the convex case, the limit free energy should be related to the solution of \eqref{e.hj}. We believe that, at least under some additional assumptions on $P_1$ and $\xi$ (see \ref{i.assume_4_product_proba} and \ref{i.assume_5_diag}), the Hopf-like representation derived in \eqref{e.main scalar} should be available for the solution of \eqref{e.hj} even when $D > 1$ and $\xi$ is possibly non-convex on $S^D_+$. Together with \cite[Conjecture~2.6]{mourrat2019parisi}, this yields a conjectural variational formula for the limiting value of the free energy when $\xi$ is possibly non-convex on $S^D_+$ (see Conjecture~\ref{c.var formula}). 

\subsection{Organization of the paper}

We start by giving a proper definition of the enriched free energy in Section~\ref{s.enriched}. In Section~\ref{s.fenchel-moreau}, we give a version of the Fenchel-Moreau theorem which holds for concave functions defined on the set of signed measures. In words, this result says that a concave function can be written as the infimum of the family of affine functions that upper bound it. In Section~\ref{s. scalar models}, using the Fenchel-Moreau theorem, we prove Theorem~\ref{t.main scalar}, the main argument is a $\sup-\inf$ interchange performed using \cite{sion1958minimax}. To do so we rely crucially on the fact that $\mcl P^\upa(\R_+) = \mcl P(\R_+)$ is a convex set. In Section~\ref{s.extension}, to compensate for the lack of convexity of $\mcl P^\upa(S^D_+)$ when $D > 1$, we show that any concave Lipschitz function $\psi: \mcl P^\upa(S^D_+) \to \R$ can be extended to a concave Lipschitz function on $\mcl P(S^D_+)$ (and even on the set of signed measures). The results of Section~\ref{s.extension} allows us to prove Theorem \ref{t.main} using similar arguments than in the proof of Theorem \ref{t.main scalar}, this is done in Section~\ref{s.vector models}. Finally, in Section~\ref{s.interpretation} we explain the link between the un-inverted formulas and the Hopf-like representation for the viscosity solution of Hamilton-Jacobi equations.

\subsection*{Acknowledgement}

I warmly thank Jean-Christophe Mourrat for motivating this project, as well as many helpful discussions and some decisive comments and ideas for Section~\ref{s.extension}. 

\section{The enriched free energy} \label{s.enriched}

Let $\mu \in \mcl P^\upa_1(S^D_+)$ be a measure with finite support, we can write 
\begin{equation} \label{e.def.finitely.supported}
    \mu  = \sum_{k = 0}^K (\zeta_{k+1} - \zeta_k) \delta_{q_k},
\end{equation}
with $K \in \mathbb{N}$, $\zeta_0,\dots,\zeta_{K+1} \in \mathbb{R}$ satisfying
\begin{equation}
    0 = \zeta_0 < \zeta_1 < \dots < \zeta_K < \zeta_{K+1} = 1,
\end{equation}
and $q_0,\dots,q_K \in {S}^D_+$ such that
\begin{equation}
    0 = q_{-1} \leq q_0 < q_1 \dots < q_{K-1} < q_K.
\end{equation}
The definition of the enriched free energy will involve a probabilistic object called the Poisson-Dirichlet cascade. We will briefly recall some properties of this object, a full definition can be found in \cite[Section~2.3]{pan}. We let 
\begin{e*}
    \mcl A = \N^0 \cup \N^1 \cup \dots \cup \N^K.
\end{e*}
We think of $\mcl A$ as a tree rooted at $\N^0 = \{\emptyset\}$, with depth $K$ and such that each non-leaf vertex has countably infinite degree. For every $k \in \{0,\dots,K\}$, and $\alpha \in \N^k \subset \mcl A$, we let $|\alpha| =  k$ denote the depth of $\alpha$. For every leaf $\alpha =(n_1,\dots,n_K) \in \N^K$, we let 
\begin{e*}
    \alpha \big|_k = (n_1,\dots,n_k),
\end{e*}
with the understanding that $\alpha \big|_0$ is the root. For every $\alpha, \alpha' \in \mcl A$ we define
\begin{e*}
    \alpha \wedge \alpha' = \sup \left\{k \in \{0,\dots K \} \big| \, \alpha \big|_k = \alpha' \big|_k \right\}.
\end{e*}
A Poisson-Dirichlet cascade is the set $(v_\alpha)_{\alpha \in \mathbb{N}^K}$ of weights of a certain random probability measure on the set $\N^K$ of leaves of the tree $\mcl A$. Those weights are constructed as follows. The children $\alpha1,\alpha2,\alpha3,\dots$ of each vertex $\alpha \in \N^k$ for $k < K$ are decorated with the values $(u_{\alpha k})_{k \geq 1}$ (arranged in decreasing order) of an independent Poisson point process with intensity measure $\zeta_{k+1}x^{-(1+\zeta_{k+1})}dx$. The weight of the leaf $\alpha \in \N^K$ is then calculated by taking the product of each of the weights associated to $\alpha \big|_k$ for $k \in \{1,\dots K\}$. Finally, the weights are normalized so that their sum is $1$. Namely, if $\alpha \in \N^K$, we have 
\begin{e*}
    v_\alpha = \frac{w_\alpha}{\sum_{\beta \in \N^K} w_\beta},
\end{e*}
where $w_\alpha = \prod_{k = 1}^K u_{\alpha \big|_k}$.

We say that a random vector $z \in \mathbb{R}^{D \times N}$ is standard Gaussian when its coordinates in an orthonormal basis form a family of real independent standard Gaussian random variables.  Let $(z_\alpha)_{\alpha \in \mathbb{N}^K}$ be a family of independent standard Gaussian vectors on $\mathbb{R}^{D \times N}$. For every $\sigma \in\mathbb{R}^{D \times N}$ and $\alpha \in \mathbb{N}^K$, we set 
\begin{equation}
    H^\mu_N(\sigma,\alpha) = \sum_{k = 0}^K (2q_k - 2q_{k-1})^{1/2} z_{\alpha |_k} \cdot \sigma,
\end{equation}
where $(2q_k - 2q_{k-1})^{1/2}$ should be understood as the square root of the symmetric positive semi-definite matrix $2q_k - 2q_{k-1}$. Here, $D \times D$ matrices act on $\mathbb{R}^{D \times N}$ via the standard multiplication of $D \times D$ matrices by $D \times N$ matrices. Alternatively, $H^\mu_N$ can be defined as the unique centered Gaussian process on $\R^D \times \N^K$ with the following covariance 
\begin{e*}
    \E H^\mu_N(\sigma,\alpha) H^\mu_N(\tau,\beta) = 2 q_{\alpha \wedge \beta} \cdot \sigma \tau^\perp.
\end{e*}
For every $t \geq 0$, we define the enriched free energy at $(t,\mu)$ by
\begin{equation} \label{e.def enriched fe}
    F_N(t,\mu) = -\frac{1}{N}  \log \int \sum_{\alpha \in \mathbb{N}^K} \exp \left( H_N^{t,\mu}(\sigma,\alpha)\right) v_\alpha \d P_N(\sigma),
\end{equation}
where
\begin{e}
    H_N^{t,\mu}(\sigma,\alpha) = \sqrt{2t} H_N(\sigma) -Nt \xi \left( \frac{\sigma \sigma^\perp}{N}\right) + H^\mu_N(\sigma,\alpha) - q_K \cdot \sigma\sigma^\perp.
\end{e}
We let $\Tilde{\E}$ denote the expectation conditionally on the randomness coming from $(H_N(\sigma))_{\sigma \in \mathbb{R}^{D \times N}}$. We define the partially and fully averaged free energies
\begin{align*}
    \Tilde{F}_N(t,\mu) =  \tilde \E [ F_N(t,\mu) ], \\
    \bar F_N(t,\mu) = \E[ F_N(t,\mu)].
\end{align*}
For every $t \geq 0$, $\Tilde{F}_N(t,\cdot)$ is Lipschitz continuous on the set of finitely supported measures in $\mcl P^\upa_\infty(S^D_+)$ \cite[Proposition~3.1]{mourrat2020free}. In particular $\bar F_N$ and $\tilde F_N$ can be extended by continuity to $\R_+ \times \mcl P^\upa_1(S^D_+)$. We let $\psi =  F_1(0,\cdot)$, that is $\psi$ is the unique Lipschitz continuous function on $\mcl P^\upa_1(S^D_+)$ such that for every finitely supported $\mu$ as in \eqref{e.def.finitely.supported}, we have 
\begin{equation} \label{e.def_psi}
    \psi(\mu) = - \mathbb{E} \log \int \sum_{\alpha \in \mathbb{N}^K} \exp \left( H^\mu_1(\sigma,\alpha) - q_K \cdot \sigma \sigma^\perp \right) v_\alpha  \d P_1(\sigma). 
\end{equation}
Let $U$ be a uniform random variable on $[0,1)$. Given a monotonically coupled measure, $\mu \in \mcl P^\upa_p(S^D_+)$, there exists a unique nondecreasing right continuous path with left limits $\msf q_\mu \in L^p([0,1),S^D)$ such that $\mu$ is the law of the random variable $X_\mu = \msf q_\mu(U)$. We let $\mcl Q(S^D_+)$ denote the set of nondecreasing right continuous paths with left limits, and we define $\mcl Q_p(S^D_+) = \mcl Q(S^D_+) \cap L^p([0,1),S^D)$. The set $\mcl Q_p(S^D_+)$ is a closed convex cone of $L^p([0,1),S^D)$, meaning that it is a closed convex subset and for every $\msf q, \msf q' \in \mcl Q_p(S^D_+)$, $t,t' \geq 0$ we have 
\begin{e*}
    t \msf q + t' \msf q' \in \mcl Q_p(S^D_+).
\end{e*}
The cone $\mcl Q_2(S^D_+)$ is embedded in the Hilbert space $L^2([0,1),S^D)$, we can define its dual cone
\begin{e*}
    \mcl Q_2(S^D_+)^* = \left\{ \msf p \in L^2([0,1),S^D) \big| \, \forall \msf q \in \mcl Q_2(S^D_+), \; \langle \msf p, \msf q\rangle_{L^2} \geq 0  \right\}.
\end{e*}
Let $\kappa \in L^2$, according to \cite[Lemma~3.5]{chenmourrat2023cavity}, we have $\kappa \in \mcl Q_2(S^D_+)^*$ if and only if for every $t \in [0,1)$,
\begin{e*}
    \int_t^1 \kappa(u) \d u \in S^D_+.
\end{e*}
Given $\msf q, \msf q' \in \mcl Q_1(S^D_+)$, we write $\msf q \leq \msf q'$ whenever for every $t \in [0,1)$,
\begin{e*}
    \int_t^1 \msf q'(u) - \msf q(u) \d u \in S^D_+.
\end{e*}
When $\msf q, \msf q' \in \mcl Q_2(S^D_+)$, we have $\msf q \leq \msf q'$ if and only if $\msf q' - \msf q \in \mcl Q_2(S^D_+)^*$. This notation can be extended to $\mcl P^\upa_1(S^D_+)$ by setting $\mu \leq \mu'$ whenever $\msf q_{\mu} \leq \msf q_{\mu'}$. We say that a function defined on a subset of $\mcl P^\upa_1(S^D_+)$ is nondecreasing when it is nondecreasing with respect to the order that we have just defined.

The dual of the closed convex cone $ \mcl Q_2(S^D_+)^*$ is the cone $\mcl Q_2(S^D_+)$ itself, more precisely we have 
\begin{e} \label{e.biduality}
     \mcl Q_2(S^D_+) = \left\{ \msf q \in L^2([0,1),S^D) \big| \, \forall \msf p \in \mcl Q_2(S^D_+)^*, \; \langle \msf p, \msf q\rangle_{L^2} \geq 0  \right\}.
\end{e}
This property will be useful to show that certain functions defined on $\mcl P^\upa_2(S^D_+)$ are nondecreasing.

The point of introducing the enriched free energy \eqref{e.def enriched fe} is \cite[Corollary~8.2]{chenmourrat2023cavity} which we state as Theorem~\ref{t.limit.free.energy} below. Roughly speaking, Theorem~\ref{t.limit.free.energy} states that the limit of the enriched free energy is the unique solution to a partial differential equation. One can recover the classical Parisi formula \eqref{e.parisi} from Theorem~\ref{t.limit.free.energy} by simply setting $\mu = \delta_0$ in \eqref{e.lim.FN}. Theorem~\ref{t.limit.free.energy} gives a new way of interpreting the Parisi formula, which does not solely rely on a variational representation. In particular, this point of view allows us to formulate a credible conjecture for the limit free energy of models whose covariance function $\xi$ is nonconvex on $S^D_+$ \cite[Conjecture~2.6]{mourrat2019parisi}. We refer to  \cite{chenmourrat2023cavity,mourrat2020nonconvex,mourrat2020free} for partial results in this direction. 
\begin{theorem}[\cite{chenmourrat2023cavity}] \label{t.limit.free.energy}
    Suppose that the function $\xi$ is convex over $S^D_+$. Then for every $t \ge 0$ and $\mu \in \mcl P^\upa_1(S^D_+)$, we have
    \begin{equation} \label{e.lim.FN}
        \lim_{N \to +\infty} \bar F_N(t,\mu) = \sup_{\nu \in \mcl P^\upa_\infty(S^D_+), \nu \ge \mu} \Ll\{ \psi(\nu) - t \E \Ll[ \xi^* \Ll( \frac{X_\nu - X_\mu}{t} \Rr)  \Rr]  \Rr\} .
    \end{equation}
    Moreover, denoting by $f(t,\mu)$ the expression above, we have that $f : \R_+ \times \mcl P^\upa_2(S^D_+) \to \R$ solves the Hamilton-Jacobi equation
    \begin{equation} \label{e.hj}
        \Ll\{
        \begin{aligned} 
            & \dr_t f - \int \xi(\dr_{\mu} f) \, \d \mu = 0 & \text{ on } \R_+ \times \mcl P^\upa_2(S^D_+) \\
            & f(0,\cdot) = \psi \quad & \text{ on } \mcl P^\upa_2(S^D_+). 
        \end{aligned}
        \Rr.
    \end{equation}
\end{theorem}
\begin{remark}
    When $D = 1$, as pointed out in \cite[Propostion~A.3]{chen2022hamilton}, we can drop the condition “$\nu \ge \mu$" in \eqref{e.lim.FN} and we have 
    \begin{equation} \label{e.lim.FN.scalar}
        \lim_{N \to +\infty} \bar F_N(t,\mu) = \sup_{\nu \in \mcl P_\infty(\R_+)} \Ll\{ \psi(\nu) - t \E \Ll[ \xi^* \Ll( \frac{X_\nu - X_\mu}{t} \Rr)  \Rr]  \Rr\} .
    \end{equation}
\end{remark}
In \eqref{e.hj}, the symbol $\partial_\mu$ should be interpreted as a derivative of transport type. Informally, this means that given any sufficiently smooth function $g : \mcl P_2^\upa(S^D_+) \to \R$, the symbol $\partial_\mu g(\mu)$ should be understood as an $S^D$-valued function defined on $S^D_+$ that is square integrable with respect to the measure $\mu$ and is such that the following holds as $\mu'$ converges to $\mu$ in $\mcl P_2^\upa(S^D_+)$,
\begin{e*}
    g(\mu') = g(\mu) + \E [\partial_\mu g(\mu)(X_\mu) \cdot (X_{\mu'} - X_\mu) ] + o\left( \E \left[ \left|X_{\mu'} - X_\mu\right|^2 \right ]^{\frac{1}{2}} \right).
\end{e*}
\begin{definition} \label{d.frechet}
    Let $\msf g : \mcl Q_2(S^D_+) \to \R$ and $\msf q \in \mcl Q_2(S^D_+)$. We say that $\msf g$ is Fréchet differentiable at $\msf q$ if there exists a unique $\msf p \in L^2([0,1),S^D)$ such that
    \begin{e*}
        \lim_{r \to 0} \sup_{\substack{\msf q' \in \mcl Q_2(S^D_+) \\ 0 < |\msf q - \msf q'|_{L^2} \leq r}} \frac{\msf g (\msf q') - \msf g (\msf q) - \langle \msf p, \msf q' - \msf q \rangle_{L^2} }{|\msf q - \msf q'|_{L^2}} = 0. 
    \end{e*}
    In this case, we say that $\msf p$ is the Fréchet derivative of $\msf g$ at $\msf q$ and we denote it $\nabla \msf g(\msf q)$.
\end{definition}

The derivative $\partial_\mu$ can be reinterpreted as a Fréchet derivative in the following way. The map $\Omega : \mu \mapsto \msf q_\mu$ yields a nonlinear isometric bijection between $\mcl P^\upa_2(S^D_+)$ and $\mcl Q_2(S^D_+)$. Given a sufficiently smooth function $g : \mcl P_2^\upa(S^D_+) \to \R$, we define $\msf g  = g \circ \Omega^{-1}$. Then formally, we have for every $u \in [0,1)$
\begin{e*}
    \partial_\mu g(\mu) (\msf q_\mu(u)) = \nabla \msf g(q_\mu)(u).
\end{e*}
In particular, setting $\msf f(t,\msf q) =  f(t,\Omega^{-1} \msf q)$, the partial differential equation \eqref{e.hj} can be seen as a special case of 
\begin{e} \label{e.hj frechet}
   \begin{cases}
       \partial_t \mathsf f - H(\nabla \mathsf f  ) = 0 \textrm{ on } (0,+\infty) \times \mcl Q_2(S^D_+) \\
       \msf f (0,\msf q) = \psi(\Omega^{-1} \msf q),
   \end{cases} 
\end{e}
with $H(\msf p ) = \int_0^1 \xi( \msf p(u)) \d u$. This is the point of view adopted in \cite{chen2022hamilton} to prove that \eqref{e.hj} is well-posed. Note that in \eqref{e.hj frechet} the nonlinearity $H$ doesn't depend directly on $\msf q$, while in \eqref{e.hj} the nonlinearity depends on $\mu$ through the integration of $\xi(\dr_{\mu} f)$ with respect to $\mu$. Despite this apparent simplification, some important properties of $f$ and $\psi$ can only be seen when they are considered as functions of $\mu$ and not $\msf q$. For example $\psi : \mcl P_2^\upa(S^D_+) \to \R$ is concave along straight lines \cite{chen2023parisipdevector}, but this doesn't imply that $\msf q \mapsto \psi(\Omega^{-1} \msf q)$ is concave, since in general,
\begin{e*}
    \msf q_{(1-\lambda) \mu_0 + \lambda \mu_1} \neq (1-\lambda) \msf q_{\mu_0} + \lambda \msf q_{\mu_1}.
\end{e*}
Finally, note that the cone $\mcl Q_2(S^D_+)$ has empty interior in $L^2$. We define $\mcl Q_\upa(S^D_+)$ as the set of paths $\msf q \in \mcl Q_2(S^D_+)$ such that there exists $c > 0$ satisfying the following for every $u \leq v$,
\begin{e*}
    \left( \msf q(v) - cv \text{Id} \right) -  \left( \msf q(u) - cu \text{Id} \right) \in S^D_+ \textrm{ and } \text{Ellipt}( \msf q(v) - \msf q(u)) \leq \frac{1}{c},
\end{e*}
where for $m \in S^D$, $\text{Ellipt}(m) \in [1,+\infty)$ denotes the ratio between the smallest and the largest eigenvalue of $m$. In practice  $\mcl Q_\upa(S^D_+)$ plays the role of the interior of $\mcl Q_2(S^D_+)$ since for every $\msf q \in \mcl Q_\upa(S^D_+)$ and every Lipschitz $\kappa : [0,1) \to S^D$, we have $\msf q +\varepsilon \kappa \in \mcl Q_2(S^D_+)$ for every $\varepsilon > 0$ small enough. 
\section{Fenchel-Moreau duality in Wasserstein space} \label{s.fenchel-moreau}
 
Let $F \subset S^D_+$ be a closed set, and $x_0 \in F$. Let $\mcl L(F,x_0)$ denote the set of Lipschitz functions $\chi : F \to \R$, equipped with the norm $|\chi|_\mcl{L} = \max\{|\chi(x_0)|,|\chi|_{\text{Lip}}\}$ where
\begin{e*}
    |\chi|_{\text{Lip}} = \sup_{\substack{x,y \in F \\ x \neq y}} \left\{ \frac{|\chi(x) - \chi(y)|}{|x-y|}\right\}.
\end{e*}
When the point $x_0$ is clear from context, we simply write $\mcl L(F)$. We let $\mcl L_{\leq 1}(F)$ denote the unit ball of $\mcl L(F)$, that is 
\begin{e*}
    \mcl L_{\leq 1}(F) = \{ \chi \in \mcl L(F) \big| \, |\chi|_{\mcl L} \leq 1 \}.
\end{e*}
Note that, according to the Arzelà-Ascoli theorem, $\mcl L_{\leq 1}(F)$ is compact with respect to the topology of local uniform convergence. We also let $\mcl L^0(F)$ denote the set of functions $\chi \in \mcl L(F)$ satisfying $\chi(x_0)= 0$ and $\mcl L_{\leq 1}^0(F) = \mcl L^0(F) \cap \mcl L_{\leq 1}(F)$. 

Given a signed Borel measure $\nu$ on $F$ with finite first moment, we define its Kantorovich-Rubinstein norm, 
\begin{equation*} \label{e.OT distance dual}
    |\nu|_{\mcl M} = \sup_{\chi \in \mcl L_{\leq 1}(F)} \left\{ \int_{F} \chi \d\nu \right\}.
\end{equation*}
We let $\mcl M_1(F)$ denote the completion with respect to $|\cdot|_{\mcl M}$ of the set of signed Borel measures on $F$ having finite first moment. The closed linear span of $\{\delta_x \big | \, x \in F \}$ is $\mcl M_1(F)$. Note that the distance $d(\nu,\nu') = |\nu - \nu'|_{\mcl M}$ induced by the norm $|\cdot|_{\mcl M}$ coincides with the optimal transport distance when restricted to $\mcl P_1(F) \times \mcl P_1(F)$ \cite[Theorem~5.10]{villani2}, in particular for every $\nu,\nu' \in \mcl P_1(F)$ we have
\begin{e*}
    |\nu-\nu'|_\mcl{M} = \inf_{\pi \in \Pi(\nu,\nu')} \left\{ \int_{F \times F} |x-y| \d \pi(x,y) \right\},
\end{e*}
where $\Pi(\nu,\nu')$ denotes the set of probability measures $\pi \in \mcl P_1(F \times F)$ such that $\pi_1 = \nu$ and $\pi_2 = \nu'$. Here and throughout, $\pi_1$ and $\pi_2$ denote the first and second marginal of the coupling $\pi$. More precisely, if $p_i : F \times F \to F$ denotes the projection on the $i$-th coordinate, we have $\pi_i = {p_{i}}_\#\pi $

\begin{proposition}
   The continuous dual of $(\mcl M_1(F), |\cdot|_{\mcl M})$ is $(\mcl L(F), |\cdot|_{\mcl L})$.
\end{proposition}

\begin{proof}
    Let $\ell$ be a continuous linear form on $(\mcl M(F), |\cdot|_{\mcl M})$, for every $x \in F$ we let $\chi(x) = \ell(\delta_x)$. For every $x,y \in F$ we have 
    \begin{e*}
        |\chi(x) - \chi(y)| = |\ell(\delta_x - \delta_y)| \leq c |\delta_x - \delta_y|_{\mcl M} = c|x-y|,
    \end{e*}
    therefore $\chi \in \mcl L(F)$. For every $x_1,\dots,x_K \in F$ and for every $\lambda_1,\dots,\lambda_K \in \R$, we have 
    \begin{e*}
        \ell \left( \sum_{k = 1}^K \lambda_k \delta_{x_k} \right) = \sum_{k = 1}^K \lambda_k \chi(x_k).
    \end{e*}
    This means that for every finitely supported $\mu \in \mcl M_1(F)$,
    \begin{e*}
        \ell(\mu) = \int \chi \d\mu.
    \end{e*}
    Since $\chi \in \mcl L(F)$, the map $\mu \mapsto \int \chi \d\mu$ defines a continuous linear form on $(\mcl M_1(F), |\cdot|_{\mcl M})$ that coincides with $\ell$ on the set of finitely supported measures. By density, we have for every $\mu \in \mcl M_1(F)$, 
    \begin{e*}
        \ell(\mu) = \int \chi \d\mu.
    \end{e*}
    We have proven that $(\mcl M_1(F), |\cdot|_{\mcl M})^*= (\mcl L(F), |\cdot|^*_{\mcl M})$, where
    \begin{e*}
        |\chi|^*_{\mcl M} = \sup_{|\mu|_{\mcl M} \leq 1} \left\{ \int \chi \d\mu \right\}.
    \end{e*}
    To conclude, let us show that for every $\chi \in \mcl L(F)$, $|\chi|^*_{\mcl M} = |\chi|_{\mcl L}$. By construction, it is clear that $|\chi|^*_{\mcl M} \leq |\chi|_{\mcl L}$. To show the converse inequality, one can plug in $\mu = \pm (\delta_x - \delta_y)/|x-y|$ and $\mu = \pm \delta_{x_0}$ in the previous display to discover that
    \begin{align*}
        |\chi|^*_{\mcl M}  &\geq \frac{|\chi(x) -\chi(y)|}{|x-y|}, \\
        |\chi|^*_{\mcl M}  &\geq  |\chi(x_0)|. \qedhere
    \end{align*}
\end{proof}

We say that a sequence of measures $(\mu_n)_n$ in $\mcl M_1(F)$ weakly converges to $\mu$ in $\mcl M_1(F)$ as $n \to +\infty$ when for every $\chi \in \mcl L(F)$,
\begin{e*}
    \lim_{n \to +\infty} \int \chi \d\mu_n = \int \chi \d\mu.
\end{e*}
We say that a function $\varphi : \mcl M_1(F) \to \R \cup \{-\infty\}$ is weakly upper semicontinuous when for every sequence $(\mu_n)_n$ that weakly converges to $\mu$ in $\mcl M_1(F)$, we have
\begin{e*}
    \limsup_{n \to + \infty} \varphi(\mu_n) \leq \varphi(\mu).
\end{e*}
Every weakly upper semi-continuous function on $\mcl M_1(F)$ is upper semi-continuous (that is, with respect to strong convergence under $|\cdot|_{\mcl M}$). The converse is not true in general, but the following result holds.
\begin{proposition} \label{p. usc implies wusc}
    Let $\varphi : \mcl M_1(F) \to \R$ be a concave function, if $\varphi$ is upper semi-continuous then $\varphi$ is weakly upper semi-continuous.
\end{proposition}
\begin{proof}
    Since $\varphi$ is concave and upper semi-continuous, the set 
    \begin{e*}
        B = \{ (x,\mu) \in \R \times \mcl M_1(F) \big|  \, \varphi(\mu) \geq x\},
    \end{e*}
    of points below the graph of $\varphi$ is closed and convex. In particular, it follows from the Hahn-Banach separation theorem, that $B$ is weakly closed \cite[Corollary~1.5]{conway1990fa}. This means that $\varphi$ is weakly upper semi-continuous. 
\end{proof}
Let us now give a statement of the Fenchel-Moreau duality in the context of the dual pair $(\mcl M_1(F), \mcl L(F))$. Usually, the Fenchel-Moreau duality is stated for convex functions, but it can be transformed into a statement about concave functions (and vice-versa) replacing each function $\varphi$ by its opposite $-\varphi$. Here the functions we are ultimately interested in in Section~\ref{s. scalar models} and \ref{s.vector models} are concave, so we choose to state the Fenchel-Moreau duality as a result for concave functions. 

Let $\varphi : \mcl M_1(F) \to \R \cup \{-\infty\}$, we define its concave conjugate $\varphi_* : \mcl L(F) \to \R \cup \{-\infty\}$ by
\begin{align*}
    \varphi_*(\chi) = \inf_{\mu \in \mcl M_1(F)} \left\{ \int \chi \d\mu - \varphi(\mu) \right\}.
\end{align*}
Similarly, given a function $\phi : \mcl L(F) \to \R \cup \{-\infty\}$ we define its concave conjugate $\phi_* : \mcl M_1(F) \to \R \cup \{-\infty\}$ by
\begin{align*}
    \phi_*(\mu) = \inf_{\chi \in \mcl L(F)} \left\{ \int \chi \d\mu - \phi(\chi) \right\}.
\end{align*}
The following theorem is a translation of \cite[Theorem~2.3.3]{convexanalysisvectorspace} in our context.

\begin{theorem}[\cite{convexanalysisvectorspace}] \label{t.fenchel moreau} 
    Let $\varphi : \mcl M_1(F) \to \R \cup \{-\infty\}$ be a function that is not identically equal to $-\infty$. Then $\varphi_{**} = \varphi$ if and only if $\varphi$ is concave and upper semicontinuous.
\end{theorem}

The concave conjugate $\varphi_*$ of $\varphi$ is defined as the infimum of a family of affine functions, hence for some $\chi \in \mcl L(F)$ we may have $\varphi_*(\chi) = -\infty$. We define,
\begin{e*}
    \text{dom}(\varphi_*) = \{ \chi \in \mcl L(F) \big| \varphi_*(\chi) > -\infty \}.
\end{e*}
For every $\mu \in \mcl M_1(F)$, we have 
\begin{e*}
    \inf_{\chi \in \mcl L(F)} \left\{ \int \chi \d\mu - \varphi_*(\chi) \right\} = \inf_{\chi \in \text{dom}(\varphi_*)} \left\{ \int \chi \d\mu - \varphi_*(\chi) \right\}.
\end{e*}
This alternative representation for $\varphi_{**}$ can be useful when the functions in $\text{dom}(\varphi_*)$ are shown to have special properties. The following proposition states that when $\varphi$ is $1$-Lipschitz, the set $\text{dom}(\varphi_*)$ is contained in the unit ball of $\mcl L(F)$.
\begin{proposition} \label{p.dual of lipschitz}
   Let $\varphi : (\mcl M_1(F), |\cdot|_{\mcl M}) \to \R$ be a $1$-Lipschitz function. For every $\chi \in \mcl L(F)$, if $|\chi|_{\mcl L} > 1$, then 
    \begin{e*}
        \varphi_*(\chi) = -\infty.
    \end{e*}   
\end{proposition}

\begin{proof} 
    Let $\chi \in \mcl L(F)$.

    \noindent Step 1. We show that if $|\chi|_{\mcl L} > 1$, then there exists $\mu' \in \mcl M_1(F)$ such that $|\mu'|_{\mcl M} < 1$ and $\int \chi \d\mu' < -1$.

    \noindent By contradiction, assume that for every $\mu \in \mcl M_1(F)$, if $|\mu|_{\mcl M} < 1$, then $\int \chi \d\mu \geq -1$. In this case, for every $\mu \in \mcl M_1(F)$ and $\varepsilon > 0$, we have 
    \begin{e*}
        \int \chi \d \left( \frac{-\mu}{(1+\varepsilon) |\mu|_{\mcl M}} \right)\geq -1.
    \end{e*}
    Letting $\varepsilon \to 0$ we obtain 
    \begin{e*}
        \int \chi \d\mu  \leq |\mu|_{\mcl M}.
    \end{e*}
    Now taking the supremum over $\mu \in \mcl M_1(F)$, this yields $|\chi|_{\mcl L} \leq 1$, a contradiction.
    
    \noindent Step 2. Conclusion.

    \noindent Let $0 \in \mcl M_1(F)$ denote the null measure. For every $\mu \in \mcl M_1(F)$, we have
    \begin{e*}
        \varphi(\mu) \geq \varphi(0) - |\mu|_\mcl{M}.
    \end{e*}
    It follows that
    \begin{e*}
        \varphi_*(\chi) \leq -\varphi(0) + \inf_{\mu \in \mcl M(F)} \left\{ \int \chi \d\mu + |\mu|_\mcl{M} \right\}.
    \end{e*}
    Assume that $|\chi|_{\mcl L} > 1$, let $\mu' \in \mcl M_1(F)$ be such that $|\mu'|_{\mcl M} < 1$ and $\int \chi \d\mu' \leq -1$ as in Step 1. We have
    \begin{e*}
        \lim_{t \to +\infty} \left( \int \chi \d (t\mu') + |t\mu'|_\mcl{M} \right) = -\infty.
    \end{e*}
    Thus, 
    \begin{e*}
        \inf_{\mu \in \mcl M_1(F)} \left\{ \int \chi \d\mu + |\mu|_\mcl{M} \right\} = -\infty,
    \end{e*}
    and it follows that $\varphi_*(\chi) = -\infty$.
\end{proof}
We use the notation $\min_{x \in X} f(x)$ to denote the value $\inf_{x \in X} f(x)$ when there exists $x_0 \in X$ such that $f(x_0)= \inf_{x \in X} f(x)$.  
\begin{corollary} \label{c.fenchel moreau for lipschitz} 
    Let $\varphi : \mcl M_1(F) \to \R \cup \{-\infty\}$ be a function that is not identically equal to $-\infty$. Assume that $\varphi$ is 1-Lipschitz with respect to $|\cdot|_{\mcl M}$ and concave, then for every $\mu \in \mcl M_1(F)$ we have 
     \begin{e*}
         \varphi(\mu) = \min_{\chi \in \mcl L_{\leq 1}(F)} \left\{ \int \chi \d\mu - \varphi_*(\chi) \right\}.
     \end{e*}
\end{corollary}

Of course, Corollary~\ref{c.fenchel moreau for lipschitz} is true if $\varphi$ is only assumed to be $L$-Lipschitz for some $L \geq 0$, in this case we need to minimize over $\mcl L_{\leq L}(F)$ rather than $\mcl L_{\leq 1}(F)$.

\begin{proof}

    \noindent Step 1. We show that
    \begin{e*}
         \varphi(\mu) = \inf_{\chi \in \mcl L_{\leq 1}(F)} \left\{ \int \chi \d\mu - \varphi_*(\chi) \right\}.
     \end{e*}

    \noindent According to Proposition~\ref{p.dual of lipschitz}, we have $\varphi_*(\chi) = -\infty$ whenever $\chi \notin \mcl L_{\leq 1}(F)$, therefore
    \begin{e*}
        \inf_{\chi \in \mcl L(F)} \left\{ \int \chi \d\mu - \varphi_*(\chi) \right\} = \inf_{\chi \in \mcl L_{\leq 1}(F)} \left\{ \int \chi \d\mu - \varphi_*(\chi) \right\}.
    \end{e*}
    Combining this with Theorem~\ref{t.fenchel moreau}, we obtain 
    \begin{e*}
         \varphi(\mu) = \inf_{\chi \in \mcl L_{\leq 1}(F)} \left\{ \int \chi \d\mu - \varphi_*(\chi) \right\}.
    \end{e*}

    \noindent Step 2. We show that the infimum in the variational formula of Step 1 is reached at some $\chi \in \mcl L_{\leq 1}$.

    \noindent For every $n \geq 1$, we let $\chi_n \in \mcl L_{\leq 1}(F)$ be such that 
    \begin{e*}
        \inf_{\chi \in \mcl L_{\leq 1}(F)} \left\{ \int \chi \d\mu - \varphi_*(\chi) \right\} \geq \int \chi_n \d\mu - \varphi_*(\chi_n) - \frac{1}{n}.
    \end{e*}
    The sequence $(\chi_n)_n$ converges locally uniformly along a subsequence $(n_k)_k$ to some $\chi \in \mcl L_{\leq 1}(F)$. In order to conclude it is enough to show that  
    \begin{e*}
        \liminf_{k \to +\infty} \left\{ \int \chi_{n_k} \d\mu - \varphi_*(\chi_{n_k}) \right\} \geq \int \chi \d\mu - \varphi_*(\chi)
    \end{e*}
    Let $\nu \in \mcl M_1(F)$, we have $\chi_{n_k} \to \chi$ pointwise on $F$ as $k \to +\infty$ and for every $x \in F$,
    \begin{e*}
        |\chi_{n_k}(x)| \leq 1+|x-x_0|.
    \end{e*}
    Therefore, by dominated convergence, we have 
    \begin{e*}
        \lim_{k \to +\infty} \int \chi_{n_k} \d\nu = \int \chi \d\nu.
    \end{e*}
    Since we have
    \begin{e*}
        \varphi_*(\chi) = \inf_{\nu \in \mcl M_1(F)} \left\{ \int \chi \d \nu - \varphi(\nu) \right\},
    \end{e*}
    it follows that $\limsup_{k \to +\infty} \varphi_*(\chi_{n_k}) \leq \varphi_*(\chi)$. Since we have 
    \begin{e*}
        \lim_{k \to +\infty} \int \chi_{n_k} \d\mu = \int \chi \d\mu,
    \end{e*}
    this concludes the proof.
\end{proof}

\section{Scalar models} \label{s. scalar models}

In this section, we focus on the case $D = 1$, and prove Theorem~\ref{t.main scalar}. Recall that in this setting $\mcl P^\upa(\R_+) = \mcl P(\R_+)$ is a convex set. Also recall that we have defined 
for every Lipschitz function $\chi : \R_+ \to \R$,  
\begin{equation*}
    S_t \chi(x) = \sup_{y \in \R_+} \left\{ \chi(y) - t \xi^* \left( \frac{y-x}{t}\right) \right\}.
\end{equation*}
The function $(t,x) \mapsto S_t \chi(x)$ can be interpreted as the unique viscosity solution of
\begin{equation} \label{e.rshj}
    \begin{cases}
        \partial_t v - \xi(\nabla v) = 0 \text{ on } (0,+\infty) \times \R_+ \\
        v(0,\cdot) = \chi.
    \end{cases}
\end{equation}

\subsection{Variational representation for the initial condition}

Let $\psi : \mcl P_1(\R_+) \to \R$, for every Lipschitz function $\chi : \R_+ \to \R$, define 
\begin{equation*}
        \psi_*(\chi) = \inf_{\mu \in \mcl P_1(\R_+)} \left\{ \int \chi \d\mu - \psi(\mu) \right\}.
\end{equation*}
We recall that $\mcl L(\R_+)$ denotes the set of Lipschitz functions $\chi : \R_+ \to \R$, we choose $x_0 = 0$ as the reference point for the norm $|\cdot|_{\mcl L}$. We also recall that $\mcl L_{\leq 1}(\R_+)$ denotes the unit ball of $\mcl L(\R_+)$ and $\mcl L^0(\R_+)$ denote the set of functions $\chi \in \mcl L(\R_+)$ satisfying $\chi(0) = 0$. We let $\mcl X$ denote the set of functions $\chi \in \mcl L_{\leq 1}(\R_+)$ that are convex and nondecreasing, and $\mcl X^0 = \mcl X \cap \mcl L^0$. Also recall that given $\mu,\nu \in \mcl P_1(\R_+)$, we say that $\mu \leq \nu$ whenever for every $t \in [0,1)$,
\begin{e} \label{e.def order}
    \int_t^1 \msf q_\nu(u) - \msf q_\mu(u)\d u \geq 0,
\end{e}
and that $\psi : \mcl P_1(\R_+) \to \R$ is said to be nondecreasing whenever we have for every $\mu,\nu \in \mcl P_1(\R_+)$,
\begin{e*}
    \mu \leq \nu \implies \psi(\mu) \leq \psi(\nu).
\end{e*}
Finally, recall that we equip $\mcl M_1(\R_+)$ with the norm $|\cdot|_{\mcl M}$ and we denote by $d$ the associated distance,
\begin{e*}
    d(\nu,\nu') = |\nu-\nu'|_{\mcl M}.
\end{e*}
We recall that we use the notation $\min_{x \in X} f(x)$ to denote the value $\inf_{x \in X} f(x)$ when there exists $x_0 \in X$ such that $f(x_0)= \inf_{x \in X} f(x)$. 
\begin{lemma} \label{l.fenchel moreau}
    Let $\psi : \mcl P_1(\R_+) \to \R$ be a $1$-Lipschitz, concave and nondecreasing function. Then, for every $\mu \in \mcl P_1(\R_+)$ we have 
    \begin{equation*}
        \psi(\mu) = \min_{\chi \in \mcl X^0} \left\{ \int \chi \d\mu - \psi_*(\chi) \right\}.
    \end{equation*}
\end{lemma}
In what follows, given $h : \R_+ \to \R$, we use $\int h(\msf q)$ as a shorthand for $\int_0^1 h(\msf q(u))\d u$.
\begin{proof}[Proof of Lemma~\ref{l.fenchel moreau}]
    For every $\mu \in \mcl M_1(\R_+)$, define 
    \begin{equation*} 
        \overline{\psi}(\mu) = \sup_{\nu \in \mcl P_1(\R_+)} \left\{ \psi(\nu) - d(\mu,\nu) \right\}.
    \end{equation*}
    The function $\overline{\psi} : \mcl M_1(\R_+) \to \R$ is 1-Lipschitz as a supremum of Lipschitz functions and is concave as a supremum of a jointly concave functional. In addition, since $\psi$ is $1$-Lipschitz, for every $\mu \in \mcl P_1(\R_+)$ we have $\overline{\psi}(\mu) = \psi(\mu)$. For every $\chi \in \mcl L(\R_+)$, we let
    \begin{e*}
        \overline{\psi}_*(\chi) = \inf_{\mu \in \mcl M_1(\R_+)} \left\{ \int \chi \d\mu - \overline{\psi}(\mu) \right\}.
    \end{e*}
    \noindent Step 1. We show that for every $\chi \in \mcl L_{\leq 1}(\R_+)$, $\overline{\psi}_*(\chi)= \psi_*(\chi)$.
   
    \noindent Let $\chi \in \mcl L_{\leq 1}(\R_+)$, the map $\mu \mapsto \int \chi \d\mu$ is $1$-Lipschitz on $\mcl M_1(\R_+)$ with respect to $|\cdot|_{\mcl M}$. In particular, for every $\nu \in \mcl P_1(\R_+)$, we have  
\begin{equation*}
    \inf_{\mu \in \mcl M_1(\R_+)} \left\{ \int \chi \d\mu + d(\nu,\mu) \right\} = \int \chi \d\nu.
\end{equation*}
Thus,
\begin{align*}
    \overline{\psi}_*(\chi) &= \inf_{\mu \in \mcl M_1(\R_+)} \left\{ \int \chi \d\mu - \overline{\psi}(\mu) \right\} \\
                           &= \inf_{\mu \in \mcl M_1(\R_+)} \inf_{\nu \in \mcl P_1(\R_+)} \left\{ \int \chi \d\mu - \psi(\nu) + d(\mu,\nu) \right\} \\
                           &= \inf_{\nu \in \mcl P_1(\R_+)} \left\{- \psi(\nu) + \inf_{\mu \in \mcl M_1(\R_+)}  \left\{ \int \chi \d\mu  + d(\mu,\nu)\right\} \right\}\\
                           &= \inf_{\nu \in \mcl P_1(\R_+)} \left\{- \psi(\nu) + \int \chi \d\nu \right\} \\
                           &= \psi_*(\chi).                     
\end{align*}

    \noindent Step 2. We show that, for every $\mu \in \mcl P_1(\R_+)$,
    \begin{equation*}
        \psi(\mu) = \min_{\chi \in \mcl L^0_{\leq 1}(\R_+)} \left\{ \int \chi \d\mu - \psi_*(\chi) \right\}.
    \end{equation*}

    \noindent According to Corollary~\ref{c.fenchel moreau for lipschitz}, we have for every $\mu \in \mcl M_1(\R_+)$,
    \begin{e*}
        \overline{\psi}(\mu) = \min_{\chi \in \mcl L_{\leq 1}(\R_+)} \left\{ \int \chi \d\mu - \overline{\psi}_*(\chi) \right\}.
    \end{e*}
    According to Step 1, for every $\chi \in \mcl L_{\leq 1}(\R_+)$, $\overline{\psi}_*(\chi) = \psi_*(\chi)$. In addition, for every $c \in \R$, we have $\psi_*(\chi+c)= \psi_*(\chi) +c$. Therefore, since $\overline{\psi}$ is an extension of $\psi$, it follows from the previous display that for every $\mu \in \mcl P_1(\R_+)$,
    \begin{equation*}
        \psi(\mu) = \min_{\chi \in \mcl L^0_{\leq 1}(\R_+)} \left\{ \int \chi \d\mu - \psi_*(\chi) \right\}.
    \end{equation*}
    \noindent Step 3. We show that in the formula of Step 2, the minimum can be taken over the set of $\chi \in \mcl L^0_{\leq 1}(\R_+)$ that are nondecreasing and convex.

    \noindent Let $\lambda$ denote the Lebesgue measure on $\R$, let $\mu \in \mcl P_1(\R_+)$ be such that the associated path $\msf q = \msf q_\mu : [0,1) \to \R_+$ is surjective, belongs to $\mcl Q_\upa(\R_+)$ and satisfies $\lambda( \msf q^{-1}(A)) = 0$ for any $\lambda$-negligible set $A \subset \R$. According to Step 2, there exists $\chi \in \mcl L^0_{\leq 1}(\R_+)$ such that
    \begin{e*}
        \psi(\mu) = \int \chi \d\mu - \psi_*(\chi).
    \end{e*}
    In addition, for every $\mu' \in \mcl P_1(\R_+)$ we have,
    \begin{e*}
        \psi(\mu') \leq \int \chi \d\mu' - \psi_*(\chi).
    \end{e*}
    Let $\msf q' \in \mcl Q_2(\R_+)$ such that $\msf q' - \msf q \in \mcl Q_2(\R_+)^*$, and let $\mu' = \text{Law}(\msf q'(U))$. By definition, we have $\mu' \geq \mu$ and,
    \begin{e*}
        \int \chi(\msf q') - \int \chi(\msf q) \geq \psi(\mu') - \psi(\mu) \geq 0.
    \end{e*}
    Now let $\kappa \in \mcl Q_2(\R_+)^*$ be a Lipschitz path, for $\varepsilon > 0$ small enough we have $\msf q +\varepsilon \kappa \in \mcl Q_2(\R_+)$, applying the previous display to $\msf q' = \msf q +\varepsilon \kappa$, we obtain  
    \begin{e*}
        \int_0^1 \frac{\chi(\msf q(u)+\varepsilon \kappa(u)) -  \chi(\msf q(u))}{\varepsilon} \d u \geq 0.
    \end{e*}
    Since $\chi$ is Lipschitz, according to Rademacher's theorem, $\chi$ is differentiable almost everywhere, so for almost every $u \in [0,1)$, we have
    \begin{e*}
        \lim_{\varepsilon \to 0} \frac{\chi(\msf q(u)+\varepsilon \kappa(u)) -  \chi(\msf q(u))}{\varepsilon}  = \nabla \chi(\msf q(u)) \kappa(u).
    \end{e*}
    In addition,
    \begin{e*}
        \left| \frac{\chi(\msf q(u)+\varepsilon \kappa(u)) -  \chi(\msf q(u))}{\varepsilon}  \right| \leq |\kappa(u)|.
    \end{e*}
    Therefore, by dominated convergence as $\varepsilon \to 0$, we have 
    \begin{e*}
        \langle \nabla \chi \circ \msf q, \kappa \rangle_{L^2} \geq 0.
    \end{e*}
    By density, the previous display holds for every $\kappa \in \mcl Q_2(\R_+)^*$ and thus $\nabla \chi \circ \msf q \in \mcl Q(\R_+)$. Since $\msf q$ is surjective, the function $\nabla \chi$ coincides almost everywhere with a nondecreasing and $\R_+$-valued function. In addition, since $\chi$ is Lipschitz, it is absolutely continuous, and we have for every $x < y$,
    \begin{e*}
        \frac{\chi(y) - \chi(x)}{y-x} = \frac{1}{y-x} \int_x^y \nabla \chi(z) \d z.
    \end{e*}
    The right-hand side in the previous display is nonnegative and given $a < x <b$, the mean value of $\nabla \chi$ on $[a,x]$ is smaller than the mean value of $\nabla \chi$ on $[x,b]$. Thus, $\chi$ satisfies 
    \begin{e*}
        0 \leq \frac{\chi(x) - \chi(a)}{x-a} \leq \frac{\chi(b) - \chi(x)}{b-x}.
    \end{e*}
    This means that $\chi$ is convex and nondecreasing. This proves,
    \begin{e*}
        \psi(\mu) = \inf_{\chi \in \mcl X^0} \left\{ \int \chi \d\mu - \psi_*(\chi) \right\}.
    \end{e*}
    In the previous display, the left and right-hand side are Lipschitz continuous and the equality holds for $\mu$ in a dense subset of $\mcl P_1(\R_+)$. Therefore, by density, the previous display holds for every $\mu \in \mcl P_1(\R_+)$. 
\end{proof}

\subsection{Viscosity solution with linear initial condition}

The goal of this section is to show Theorem~\ref{t.viscosity with linear initial condition} below. The reader only interested in un-inverted formulas for $\lim_{N \to +\infty} \bar F_N(t,\delta_0)$ but not for $\lim_{N \to +\infty} \bar F_N(t,\mu)$ can skip this subsection and directly go to Subsection~\ref{ss.scalar proof}, replacing the content of Theorem~\ref{t.viscosity with linear initial condition} by the following elementary formula,
\begin{e} \label{e.alternative}
    \sup_{\mu \in \mcl P_\infty(\R_+)} \left\{ \int \chi \d\mu - t \int \xi^* \left( \frac{\cdot}{t} \right) \d\mu \right\} =  \sup_{x \in \R_+} \left\{ \chi(x) - t  \xi^* \left( \frac{x}{t} \right)\right\}.
\end{e}
Let $(\mcl H, \langle \cdot, \cdot \rangle_{\mcl H})$ be a Hilbert space. Given $\mcl C \subset \mcl H$, we say that $\mcl C$ is a closed convex cone when $\mcl C$ is a closed set in $\mcl H$ and for every $x,x' \in \mcl C$ and $t,t' \geq 0$, we have 
\begin{e*}
    tx + t'x' \in \mcl C.
\end{e*}
Given a closed convex cone $\mcl C$ in $\mcl H$, we define its dual cone by
\begin{e}
    \mcl C^* = \{y \in \mcl H \big| \, \forall x \in \mcl C, \; \langle x,y \rangle_{\mcl H} \geq 0 \}.
\end{e}
For $x,x' \in \mcl H$, we say that $x \le x'$ when $x'-x \in \mcl C^*$ and we say that $\msf g : \mcl C \to \R$ is $\mcl C^*$-nondecreasing whenever for every $x,x' \in \mcl C$, we have 
\begin{e*}
    x \leq x' \implies \msf g(x) \leq \msf g(x').
\end{e*}
We define $\mcl V(\mcl C)$ as the set of functions $V : \R_+ \times \mcl C \to \R$ such that, for some $L \geq 0$ the following holds. For every $t \geq 0$, $V(t,\cdot)$ is $\mcl C^*$-nondecreasing and $L$-Lipschitz, and 
\begin{e*}
    \sup_{\substack{t > 0 \\ x \in \mcl C }} \left\{ \frac{V(t,x) - V(0,x)}{t} \right\} < +\infty.
\end{e*}
Recall the notion of Fréchet derivative $\nabla$ from Definition~\ref{d.frechet}, also recall that given a Fréchet differentiable function $\msf g : \mcl Q_2(\R_+) \to \R$, we have for every $\msf q \in \mcl Q_2(\R_+)$, $\nabla \msf g(\msf q) \in L^2([0,1),\R)$. In particular expressions of the form $\int h(\nabla \msf g(\msf q))$ should be understood as $\int_0^1 h(\nabla \msf g(\msf q)(u))\d u$. The notion of viscosity solution for \eqref{e.linear hj} appearing in Theorem~\ref{t.viscosity with linear initial condition} below is introduced in details in \cite[Definition~1.4]{chen2022hamilton}. 
\begin{theorem} \label{t.viscosity with linear initial condition}
     Let $\chi : \R_+ \to \R$ a Lipschitz, nondecreasing and convex function. The function $(t,\msf q) \mapsto \int S_t \chi(\msf q)$ belongs to $\mcl V(\mcl Q_2(\R_+))$ and is the unique viscosity solution of 
     \begin{equation} \label{e.linear hj}
         \begin{cases}
             \partial_t V - \int \xi( \nabla V) = 0 \text{ on } (0,+\infty) \times \mcl Q_2(\R_+)\\
        V(0,\msf q) = \int \chi(\msf q) \text{ on } \mcl Q_2(\R_+).
         \end{cases}
     \end{equation}
\end{theorem}

To prove Theorem~\ref{t.viscosity with linear initial condition}, we are going to use a notion of differentiability that is weaker than Fréchet differentiability.

\begin{definition}
    Let $\msf g : \mcl Q_2(S^D_+) \to \R$ and $\msf q \in \mcl Q_2(S^D_+)$. We say that $\msf g$ is Gateaux differentiable at $\msf q$ if the following conditions hold.
    
    \begin{enumerate}
        \item For every $\kappa \in L^2([0,1),S^D)$ such that $\msf q + \varepsilon \kappa \in \mcl Q_2(S^D_+)$ for $\varepsilon > 0$ small enough, the following limit exists and is finite,
        \begin{e*}
            \msf g '(\msf q,\kappa) = \lim_{\varepsilon \to 0^+} \frac{\msf g (\msf q + \varepsilon \kappa) -\msf g (\msf q)}{\varepsilon}.
        \end{e*}

        \item There exists a unique $\msf p \in L^2([0,1),S^D)$ such that for every $\kappa \in L^2([0,1),S^D)$ such that $\msf q + \varepsilon \kappa \in \mcl Q_2(S^D_+)$ for $\varepsilon > 0$ small enough, we have 
        \begin{e*}
             \msf g'(\msf q,\kappa) = \langle \msf p, \kappa \rangle_{L^2}.
        \end{e*}
    \end{enumerate}    
     In this case, we say that $\msf p$ is the Gateaux derivative of $\msf g$ at $\msf q$ and we denote it $\nabla \msf g(\msf q)$.
\end{definition}

Since a Fréchet differentiable function is also Gateaux differentiable and both derivatives are equal, there is no harm in using the symbol $\nabla$ to denote both the Gateaux and the Fréchet derivative. The Gateaux derivative allows us to characterize differentiable nondecreasing functions, roughly speaking a function  $\msf g : \mcl Q_2(S^D_+) \to \R$ is $\mcl Q_2(\R_+)^*$-nondecreasing if and only if for every $\msf q \in \mcl Q_2(\R_+)$, we have $\nabla \msf g(\msf q) \in \mcl Q_2(\R_+)$. In practice, this characterization is not exactly true at points in $\mcl Q_2(\R_+) - \mcl Q_\upa(\R_+)$ since the set of admissible directions at those points may not be rich enough.

\begin{proposition} \label{p.characterization of nondecreasing linear}
     Let $\chi : \R_+ \to \mathbb{R}$ be a Lipschitz function. The function $X : \msf q \mapsto \int_0^1 \chi(\msf q(u)) \d u$ is $\mcl Q_2(\R_+)^*$-nondecreasing if and only if $\chi$ is nondecreasing and convex.
\end{proposition}
\begin{proof} 
    We start by showing the desired equivalence under the additional assumption that $\chi$ is differentiable, and then we proceed by approximation.

    \noindent Step 1. We assume that $\chi$ is differentiable on $\R_+$ and we show that $X$ is Gateaux differentiable on $\mcl Q_2(\R_+)$ and satisfies 
    \begin{equation*}
        \nabla X(\msf q) = \nabla \chi \circ \msf q.
    \end{equation*}
    \noindent Fix $\msf q \in \mcl Q_2(\R_+)$, let $\kappa \in L^2$ such that for $\varepsilon > 0$ small enough $\msf q + \varepsilon \kappa \in \mcl Q_2(\R_+)$. For every $u \in [0,1)$,  
    \begin{equation*}
        \left| \frac{\chi(\msf q(u)+\varepsilon \kappa(u)) -\chi(\msf q(u))}{\varepsilon} \right| \leq |\chi |_\text{Lip} \kappa(u).
    \end{equation*}
    and we have as $\varepsilon \to 0$,
    \begin{equation*}
       \frac{\chi(\msf q(u)+\varepsilon \kappa(u)) -\chi(\msf q(u))}{\varepsilon} \to \nabla \chi (\msf q(u)) \kappa(u).
    \end{equation*}
    By dominated convergence, it follows that 
    \begin{equation*}
        X(\msf q+\varepsilon \kappa) = X(\msf q) + \varepsilon \int_0^1  \nabla \chi (\msf q(u)) \kappa(u) \d u + o(\varepsilon).
    \end{equation*}
    So $X$ is Gateaux differentiable at $\msf q$ and $\nabla X(\msf q) = \nabla \chi \circ \msf q  \in L^\infty([0,1),\R)$. 

    \noindent Step 2. We assume that $\chi$ is differentiable on $\R_+$ and we show that if $\chi$ is nondecreasing and convex on $\R_+$, then $X$ is $\mcl Q_2(\R_+)^*$-nondecreasing on $\mcl Q_2(\R_+)$.

    \noindent If $\chi$ is nondecreasing and convex, then $\nabla \chi : \R \to \R_+$ is nondecreasing and $\R_+$-valued. In particular, for every $\msf q \in \mcl Q_2(\R_+)$, we have 
    \begin{equation*}
        \nabla X(\msf q) = \nabla \chi \circ \msf q \in \mcl Q_2(\R_+).
    \end{equation*}
    Thus, for every $\msf q, \msf q' \in \mcl Q_2(\R_+)$ such that $\msf q' - \msf q \in \mcl Q_2(\R_+)^*$, we have 
    \begin{equation*}
        X(\msf q') - X(\msf q) = \int_0^1 \langle \nabla X(\lambda \msf q' +(1-\lambda) \msf q),\msf q'- \msf q \rangle_{L^2} \d \lambda \geq 0.
    \end{equation*}

    \noindent Step 3.  We assume that $\chi$ is differentiable on $\R_+$ and we show that if $X$ is $\mcl Q_2(\R_+)^*$-nondecreasing on $\mcl Q_2(\R_+)$, then $\chi$ is nondecreasing and convex on $\R_+$.
    
    \noindent Let $\msf q \in \mcl Q_\upa(\R_+) \cap L^2$, for every smooth function $\kappa \in \mcl Q_2(\R_+)^*$ and every $\varepsilon > 0$ small enough we have $\msf q+ \varepsilon \kappa \in \mcl Q_2(\R_+)$ and 
    \begin{e*}
        \frac{X(\msf q+\varepsilon \kappa) - X(\msf q)}{\varepsilon} \geq 0.
    \end{e*}
    Letting $\varepsilon \to 0$, we obtain 
    \begin{e*}
       \langle \nabla \chi \circ \msf q, \kappa \rangle_{L^2} \geq 0. 
    \end{e*}
    By density, this last display is in fact valid for all $\kappa \in \mcl Q_2^*$. Therefore $\nabla \chi \circ \msf q \in \mcl Q_2(\R_+)$. This ensures that $\nabla \chi$ is $\R_+$-valued and nondecreasing, therefore $\chi$ is convex and nondecreasing.

    %About the density thing, firstly it used page 36 of crit_points of Chen and Mourrat. Second, given $\kappa \in \mcl Q_2^*$ one can extend it by $0$ outside of $[0,1)$ and define \kappa_\varepsilon(u) = \int_{u}^{u+\varepsilon} \kappa(v) \eta_\varepsilon(v-u) \d v$. This defines a sequence of smooth path approximating \kappa and oen can easly check that \int_t^1 \kppa_\varepsilon(u) du \geq 0.
    
    \noindent Step 4. Approximation by differentiable functions.

    \noindent We let $\eta \in \mcl C^\infty(\R_+,\R)$ be a smooth function supported on $[1,2]$ that takes nonnegative values and such that 
    \begin{equation*}
        \int_{\R_+} \eta(x) \d x = 1.
    \end{equation*}
    For every $\varepsilon > 0$, let $\eta_\varepsilon(x) = \frac{1}{\varepsilon} \eta \left( \frac{x}{\varepsilon} \right)$. We define,
    \begin{equation*}
        \chi_\varepsilon(x) = \int_{\R_+} \chi(x+y) \eta_\varepsilon(y) \d y,
    \end{equation*}
    For every $x \in \R_+$, we have 
    \begin{align*}
        |\chi_\varepsilon(x) - \chi(x)| &= \left| \int_{\R_+} \chi(x+y) \eta_\varepsilon(y) \d y - \int_{\R_+} \chi(x) \eta_\varepsilon(y) \d y\right| \\
                                        &\leq \int_{\R_+} |\chi(x+y) - \chi(x)| \eta_\varepsilon(y) \d y \\
                                        &\leq \int_{\R_+} |\chi |_\text{Lip} y \eta_\varepsilon(y) \d y \\
                                        &= \varepsilon |\chi |_\text{Lip} \int_{\R_+} y \eta(y) \d y.        
    \end{align*}  
    Therefore as $\varepsilon \to 0$, $\chi_\varepsilon \to \chi$ uniformly on $\R_+$. In addition, for every $x \in \R_+$ and $h \in (-\varepsilon,\varepsilon)$,
    \begin{e*}
        \frac{\chi_\varepsilon(x+h)-\chi_\varepsilon(x)}{h} = \int_x^{x+2\varepsilon} \chi(z) \frac{\eta_\varepsilon(z-x-h)-\eta_\varepsilon(z-x)}{h} \d z.
    \end{e*}
    We have, uniformly over $z \in [x,x+2\varepsilon]$,
    \begin{e*}
       \lim_{h \to 0} \frac{\eta_\varepsilon(z-x-h)-\eta_\varepsilon(z-x)}{h} = \nabla \eta_\varepsilon(z-x).
    \end{e*}
    Therefore, $\chi_\varepsilon$ is differentiable at $x$ and 
    \begin{e*}
        \nabla \chi_\varepsilon(x)  = \int_x^{x+2\varepsilon} \chi(z) \nabla \eta_\varepsilon(z-x) \d z.
    \end{e*}
    Thus, the sequence $(\chi_\varepsilon)_{\varepsilon}$ is a sequence of Lipschitz differentiable functions on $\R_+$ that converge uniformly on $\R_+$ towards $\chi$.
    
    \noindent Step 5. Conclusion.

    \noindent Let $(\chi_\varepsilon)_{\varepsilon}$ be the sequence built in Step 4. Set 
    \begin{equation*}
        X_\varepsilon(\msf q) = \int_0^1 \chi_\varepsilon(\msf q(u)) \d u.
    \end{equation*}
    We have,
    \begin{equation*}
        X_\varepsilon(\msf q) = \int_{\R_+} X(\msf q+y) \eta_\varepsilon(y) \d y,
    \end{equation*}
    where $\msf q + y$ denotes the path obtained by adding the constant $y \in \R_+$ to the values of $\msf q$.  According to Step 4, $\chi_\varepsilon$ is differentiable on $\R_+$. 
    
    If $\chi$ is nondecreasing and convex, then so is $\chi_\varepsilon$. According to Step 2, in this case $X_\varepsilon$ is $\mcl Q_2(\R_+)^*$-nondecreasing and letting $\varepsilon \to 0$, we discover that $X$ is $\mcl Q_2(\R_+)^*$-nondecreasing. 
    
    Conversely, if $X$ is $\mcl Q_2(\R_+)^*$-nondecreasing, then so is $X_\varepsilon$, using Step 3 we discover that $\chi_\varepsilon$ is convex and nondecreasing, and letting $\varepsilon \to 0$ we obtain that $\chi$ is convex and nondecreasing.
\end{proof}

\begin{proposition} \label{p. int S_t chi is nondecreasing}
    Let $\chi : \R_+ \to \R$ be a Lipschitz, nondecreasing and convex function. For every $t \geq 0$, the function 
    \begin{equation*}
        (t,\msf q) \mapsto \int_0^1 S_t \chi(\msf q(u)) \d u
    \end{equation*}
    belongs to $\mcl V(\mcl Q_2(\R_+))$.
\end{proposition}
\begin{proof}
    Since $\chi$ is convex, $S_t \chi$ admits the Hopf representation \cite[Proposition~6.3]{chen2023viscosity} that is, for every $x \in \R_+$, we have,
    \begin{e*}
        S_t \chi(x) = \sup_{y \in \R_+} \left\{ xy - \chi^*(y) +t \xi(y) \right\}.
    \end{e*}
    In particular, $S_t \chi$ is the supremum of a family of affine functions of $x$, so $S_t \chi$ is convex on $\R_+$. By definition $S_t \chi \in \mcl V(\R_+)$, so $S_t \chi$ is Lipschitz and nondecreasing. It then follows from Proposition \ref{p.characterization of nondecreasing linear} that $\msf q \mapsto \int_0^1 S_t \chi(\msf q(u)) \d u$ is $\mcl Q_2(\R_+)^*$-nondecreasing and thus $(t,\msf q) \mapsto \int_0^1 S_t \chi(\msf q(u)) \d u$ belongs to $\mcl V(\mcl Q_2(\R_+))$. 
\end{proof}
For every $j \geq 1$, define 
\begin{equation}
    \mcl Q^j(\R_+) =\left\{ (x_i)_{1 \leq i \leq j} \in \R_+^j \big| \,  x_{i}-x_{i-1} \in \R_+ \right\}.
\end{equation}
The cone $\mcl Q^j(\R_+)$ is the set of nondecreasing sequences of $\R_+$ with $j$ terms. It is embedded in $\R^j$ and we equip it with the normalized Euclidean scalar product,
\begin{e*}
    \langle x,y\rangle_j = \frac{1}{j} \sum_{i = 1}^j x_i y_i.
\end{e*}
%
% As previously, we can define the dual cone $\mcl Q^j(\R_+)^*$ of $\mcl Q^j(\R_+)$ by 
% %
% \begin{e}
%     \mcl Q^j(\R_+)^* = \{ y \in \R^j \big| \forall x \in \mcl Q^j(\R_+), \; \langle x,y \rangle_j \geq 0 \}.
% \end{e}
% %
% In particular, for $x,x' \in \R^j$, we will say that $x \leq x'$ whenever $x' - x  \in \mcl Q^j(\R_+)^*$ and a function $\varphi : \mcl Q^j(\R_+) \to \R$ is said to be nondecreasing when for every $x \leq x'$ we have $\varphi(x) \leq \varphi(x')$.
%
One can think of $\mcl Q^j(\R_+)$  as a finite dimensional approximation of $\mcl Q_1(\R_+)$. Given $x \in  \mcl Q^j(\R_+)$, we define a path $l_j x \in \mcl Q_1(\R_+)$ by setting 
\begin{equation*}
    l_j x = \sum_{i = 1}^j x_i \mathbf{1}_{\left[\frac{i-1}{j},\frac{i}{j}\right)}.
\end{equation*}
Conversely, given a path $\msf q \in \mcl Q_1(\R_+)$, we define $p_j \msf q \in \mcl Q^j(\R_+)$ by setting 
\begin{equation*}
    (p_j \msf q)_i = j \int_{\frac{i-1}{j}}^\frac{i}{j} \msf q(u) \d u.
\end{equation*}
The linear maps $l_j$ and $p_j$ are adjoint in the following sense. 
\begin{equation*}  
    \left \langle \msf q,l_jx \right\rangle_{L^2} =  \left \langle p_j \msf q,x \right\rangle_{j}.
\end{equation*}
It follows that for every $\mcl Q_2(\R_+)^*$-nondecreasing function $\varphi : \mcl Q_2(\R_+) \to \R$, the function $\varphi \circ l_j$ is $(\mcl Q^j(\R_+))^*$-nondecreasing on $\mcl Q^j(\R_+)$. We also define for every $x \in \R^j$,
\begin{equation*}
    H_j(x) = \frac{1}{j} \sum_{i =1}^j \xi(j x_i),
\end{equation*}
the function $H_j : \R^j \to \R$ is locally Lipschitz, $\R_+^j$-nondecreasing on $\R_+^j$ and $(\mcl Q^j(\R_+))^*$-nondecreasing on $\mcl Q^j(\R_+)$. Given $X \subset \R^j$, we let $\mathrm{int}(X)$ denote the interior of $X$.
\begin{proposition} \label{p.visco of hj approx extended}
    Let  $\chi : \R_+ \to \R$ be a Lipschitz, nondecreasing and convex function. The function $(t,x) \mapsto \frac{1}{j} \sum_{i=1}^j S_t \chi(x_i)$ belongs to $\mcl V(\R_+^j)$ and is the unique viscosity solution of 
    \begin{equation} \label{e.hj_approx_extended}
        \begin{cases}
        \partial_t \overline v_j - H_j(\nabla \overline v_j) = 0 \text{ on } (0,+\infty) \times \mathrm{int}(\R_+^j)\\
         \overline v_j(0,x) = \frac{1}{j} \sum_{i=1}^j \chi(x_i) \text{ on } \R_+^j.
        \end{cases}
    \end{equation}
\end{proposition}
%
% \begin{remark}
%     Surprisingly, for proving Proposition~\ref{p.visco of hj approx extended}, the natural idea which would consist in trying to prove that the function $(s,y) \mapsto \frac{1}{j} \sum_{i=1}^j S_s \chi(y_i)$ satisfies \cite[Definition~1.1]{chen2023viscosity} doesn't seem to work. If one considers a smooth function $\phi$ such that $(s,y) \mapsto \frac{1}{j} \sum_{i=1}^j S_s \chi(y_i) - \phi(s,y)$ has a local extremum at $(t,x)$, an attempt to prove Proposition~\ref{p.visco of hj approx extended} could be to try to justify that the function $(s,y_i) \mapsto S_t \chi(y_i) - j\phi(t,x_1,\dots,y_i,\dots,x_j)$ has a local extremum at $(t,x_i)$, since every term in the sum $\sum_{i=1}^j S_s \chi(y_i)$ depends on $t$, this is not so in general. 
% \end{remark}
 %
\begin{proof} 
    The initial condition $\varphi_j : x \mapsto \frac{1}{j} \sum_{i=1}^j \chi(x_i)$ is convex on $\R_+^j$. We have,
    \begin{align*}
        \varphi_j^*(y) &= \sup_{x \in \mcl \R_+^j} \left\{ \langle x,y \rangle_j - \frac{1}{j}\sum_{i=1}^j \chi(x_i)\right\} \\
                       &= \sup_{x_1 \in \mcl \R_+} \dots \sup_{x_j \in \mcl \R_+} \left\{ \frac{1}{j} \sum_{i=1}^j x_iy_i -  \chi(x_i)\right\} \\
                       &= \frac{1}{j} \sum_{i=1}^j \sup_{x_i \in \mcl \R_+} \left\{ x_iy_i -  \chi(x_i)\right\} \\
                       &= \frac{1}{j} \sum_{i=1}^j \chi^*(y_i).
    \end{align*}
    According to \cite[Proposition~6.3]{chen2023viscosity}, the unique viscosity solution $\overline v_j$ of \eqref{e.hj_approx_extended} admits the Hopf representation, 
    \begin{e*}
        \overline v_j(t,x) = \sup_{y \in \R_+^j} \left\{ \langle x,y \rangle_j - \varphi_j^*(y) +tH_j(y)  \right\}.
    \end{e*}
    Thus,
    \begin{align*}
        \overline v_j(t,x) &= \sup_{y \in \R_+^j} \left\{ \frac{1}{j} \sum_{i=1}^j \left( x_iy_i - \chi^*(x_i) +t\xi(y_i) \right)  \right\} \\
                           &= \sup_{y_1 \in \R_+} \dots \sup_{y_j \in \R_+} \left\{ \frac{1}{j} \sum_{i=1}^j \left( x_iy_i - \chi^*(x_i) +t\xi(y_i) \right) \right\} \\
                           &= \frac{1}{j} \sum_{i=1}^j \sup_{y_i \in \R_+} \left\{ x_iy_i - \chi^*(x_i) +t\xi(y_i) \right\} \\
                           &= \frac{1}{j} \sum_{i=1}^j S_t \chi(x_i).                            
    \end{align*}
    By definition, $\overline v_j \in \mcl V(\R_+^j)$, this concludes the proof.
\end{proof}
\begin{proposition} \label{p.viscosity_restriction}
    Let  $\chi : \R_+ \to \R$ be a Lipschitz, nondecreasing and convex function. The function $(t,x) \mapsto \frac{1}{j} \sum_{i=1}^j S_t \chi(x_i)$ belongs to $\mcl V(\mcl Q^j(\R_+))$ and is the unique viscosity solution of 
    \begin{equation} \label{e.hjapprox}
        \begin{cases}
        \partial_t v_j - H_j(\nabla v_j) = 0 \text{ on } (0,+\infty) \times \mathrm{int}(\mcl Q^j(\R_+)) \\
        u_j(0,x) = \frac{1}{j} \sum_{i=1}^j \chi(x_i) \text{ on } \mcl Q^j(\R_+).
        \end{cases}
    \end{equation}
\end{proposition}

\begin{proof}
    According to \cite[Theorem 1.2]{chen2023viscosity}, \eqref{e.hjapprox} admits a unique viscosity solution in $\mcl V(\mcl Q^j(\R_+))$. Therefore, according to Proposition~\ref{p.visco of hj approx extended}, it is enough to check that $\overline{v}_j \big|_{\R_+ \times \mcl Q^j(\R_+)}$ is a viscosity solution of $\eqref{e.hjapprox}$ and belongs to $\mcl V(\mcl Q^j(\R_+))$. Let $(t,x) \in \R_+ \times   \mathrm{int}(\mcl Q^j(\R_+))$ and $\phi : (0,\infty) \times \mathrm{int}(\mcl Q^j(\R_+)) \to \R$ be a smooth function. Assume that $\overline{v}_j \big|_{\R_+ \times \mcl Q^j(\R_+)} - \phi$ has a local maximum at $(t,x)$. Up to modifying $\phi$ outside a ball of small radius around $(t,x)$, we may assume that $\phi$ is the restriction to $(0,\infty) \times \mathrm{int}(\mcl Q^j(\R_+))$ of a smooth function $\overline{\phi} : (0,\infty) \times \R^j \to \R$. By construction $\overline{v}_j - \overline{\phi}$ coincide with $\overline{v}_j \big|_{\R_+ \times \mcl Q^j(\R_+)} - \phi$ on a neighborhood of $(t,x)$ and thus has a local maximum at $(t,x)$. Since $\overline{v}_j$ is the viscosity solution of \eqref{e.hj_approx_extended}, we have,
    \begin{equation*}
       \left( \partial_t \overline{\phi} - H_j( \nabla \overline{\phi}) \right)(t,x) \leq 0.
    \end{equation*}
    Since $\phi$ and $\overline{\phi}$ coincide on a neighborhood of $(t,x)$, the previous display remains true if we replace $\overline{\phi}$ by $\phi$. We reach similar conclusions if we assume that $\overline{v}_j \big|_{\R_+ \times \mcl Q^j(\R_+)} - \phi$ has local minimum at $(t,x)$. This proves that $\overline{v}_j \big|_{\R_+ \times \mcl Q^j(\R_+)}$ is a viscosity solution of  \eqref{e.hjapprox}.
    
    In addition, according to Proposition~\ref{p.visco of hj approx extended}, we have for every $(t,x) \in \R_+ \times \mcl Q^j(\R_+)$,
    \begin{e*}
        \overline v_j(t,x) = \frac{1}{j} \sum_{i = 1}^j S_t \chi(x_i) = \int S_t \chi(l_j x).
    \end{e*}
    Therefore, according to Proposition~\ref{p. int S_t chi is nondecreasing}, we have $\overline{v}_j \big|_{\R_+ \times \mcl Q^j(\R_+)} \in \mcl V(\mcl Q^j(\R_+))$ which concludes. 
\end{proof}
\begin{proof}[Proof of Theorem \ref{t.viscosity with linear initial condition}]
    According to Proposition~\ref{p. int S_t chi is nondecreasing}, $(t,\msf q) \mapsto \int S_t \chi(\msf q)$ belongs to $\mcl V(\mcl Q_2(\R_+))$. Let $(t,\msf q) \in \R_+ \times \mcl Q_2(\R_+)$, according to \cite[Theorem~4.6~(1)]{chen2022hamilton}, the value of the unique viscosity solution of \eqref{e.linear hj} at $(t,\msf q)$ is the limit of $v_j(t,p_j \msf q)$ as $j \to +\infty$, where $v_j$ is the unique viscosity solution of \eqref{e.hjapprox}. According to Proposition~\ref{p.viscosity_restriction}, we have 
    \begin{e*}
        v_j(t,p_j \msf q) = \int S_t \chi(l_jp_j \msf q).
    \end{e*}
    Since $x \mapsto S_t \chi(x)$ is Lipschitz, we have 
    \begin{equation*}
        \left|\int S_t \chi(l_jp_j \msf q) - \int S_t \chi(\msf q)\right| \leq c|l_jp_j \msf q - \msf q|_{L^1}.
    \end{equation*}
    Finally, according to \cite[Lemma~3.3~(7)]{chen2022hamilton} we have $\lim_{j \to +\infty} |l_jp_j \msf q - \msf q|_{L^1} = 0$. This concludes the proof.
\end{proof}

\subsection{Proof of the un-inverted Parisi formula} \label{ss.scalar proof}

\begin{proof}[Proof of Theorem \ref{t.main scalar}]
    Let $(t,\mu) \in \R_+ \times \mcl P_\infty(\R_+)$, and
    \begin{e*}
        f(t,\mu) = \lim_{N \to +\infty} F_N(t,\mu).
    \end{e*}
    According to Theorem~\ref{t.limit.free.energy} we have
    \begin{equation*}
        f(t,\mu) = \sup_{\nu \in \mcl P_\infty(\R_+)} \left\{ \psi(\nu) - t \E \xi^* \left( \frac{X_\nu - X_\mu}{t} \right) \right\}.
    \end{equation*}
    We let $\Pi(\mu,\nu)$ denote the set of probability measures $\pi \in \mcl P_\infty(\R_+ \times \R_+)$ such that $\pi_1 = \mu$ and $\pi_2 = \nu$. According to \cite[Proposition~2.5]{mourrat2020free}, we have
    \begin{equation}
        \E \xi^* \left( \frac{X_\nu - X_\mu}{t} \right) = \inf_{ \pi \in \Pi(\nu,\mu) } \int \xi^* \left( \frac{y-x}{t} \right) \d \pi(x,y).
    \end{equation}
    Define
    \begin{e*}
        \Pi(\mu,\cdot) = \bigcup_{\nu \in \mcl P_\infty(\R_+)}  \Pi(\mu,\nu).
    \end{e*}
    It follows respectively from \cite[Proposition~3.6]{chenmourrat2023cavity}, \cite[Corollary~5.2]{chenmourrat2023cavity} and \cite{auffinger2015parisi}, that $\psi$ is nondecreasing, $1$-Lipschitz and concave. Therefore, according to Lemma~\ref{l.fenchel moreau} we have
    \begin{align*}
        f(t,\mu) &= \sup_{\pi \in \Pi(\mu,\cdot)} \left\{ \psi(\pi_2) - t \int \xi^* \left( \frac{y-x}{t} \right) \d\pi(x,y) \right\} \\
                 &= \sup_{\pi \in \Pi(\mu,\cdot)} \inf_{\chi \in \mcl X^0} \left\{ \int \chi \d\pi_2 - \psi_*(\chi) - t \int \xi^* \left( \frac{y-x}{t} \right) \d\pi(x,y) \right\} \\
                 &= \sup_{\pi \in \Pi(\mu,\cdot)} \inf_{\chi \in \mcl X^0} \left\{ \int \chi \d\pi_2 - \psi_*(\chi) - t \int \xi^* \left( \frac{y-x}{t} \right) \d\pi(x,y) \right\}. 
    \end{align*}
The sets $\Pi(\mu,\cdot)$ and $\mcl X^0$ are convex. In addition, according to the Arzelà-Ascoli theorem, $\mcl X^0$ is compact with respect to the topology of local uniform convergence. For every $\pi \in \Pi(\mu,\cdot)$, the map $\chi \mapsto \int \chi \d\pi_2 - \psi_*(\chi)$ is lower semi-continuous on $\mcl X^0$ with respect to the topology of local uniform convergence. Similarly since $\xi^*$ is lower semi-continuous, according to Portmanteau's theorem for every $\chi \in \mcl X^0$, the map $\pi \mapsto \int \chi \d\pi_2 - t \int \xi^* \left( \frac{y-x}{t} \right) \d\pi(x,y)$ is upper semi-continuous with respect to convergence under the optimal transport distance. Thus, according to \cite[Corollary~3.4]{sion1958minimax}, we can interchange sup and inf in the previous display to obtain
    \begin{align*}
                f(t,\mu) &= \inf_{\chi \in \mcl X^0} \left\{ - \psi_*(\chi) + \sup_{\pi \in \Pi(\mu,\cdot)} \left\{ \int \chi \d\pi_2 - t \int \xi^* \left( \frac{y-x}{t} \right) \d\pi(x,y) \right\} \right\} \\
                 &= \inf_{\chi \in \mcl X^0} \left\{ - \psi_*(\chi) + \sup_{\pi \in \Pi(\mu,\cdot)} \left\{ \int \chi \d\pi_2 - t \int \xi^* \left( \frac{y-x}{t} \right) \d\pi(x,y) \right\} \right\} \\
                 &= \inf_{\chi \in \mcl X^0} \left\{ - \psi_*(\chi) + \sup_{\nu \in \mcl P_\infty(\R_+)} \left\{ \int \chi \d\nu - t \E \xi^* \left( \frac{X_\nu - X_\mu}{t} \right) \right\} \right\} \\
    \end{align*} 
    According to \cite[Theorem~4.6~(2) and Proposition~A.3]{chen2022hamilton}, the viscosity solution of \eqref{e.linear hj} has the Hopf-Lax representation and it follows from Theorem~\ref{t.viscosity with linear initial condition} that 
    \begin{e*}
        \sup_{\nu \in \mcl P_\infty(\R_+)} \left\{ \int \chi \d\nu - t \E \xi^* \left( \frac{X_\nu - X_\mu}{t} \right) \right\} = \int S_t \chi \d\mu.
    \end{e*}
    Thus, from the two previous displays or simply from \eqref{e.alternative} if $\mu = \delta_0$, we have  
    \begin{align*}
        f(t,\mu) &= \inf_{\chi \in \mcl X^0} \left\{ - \psi_*(\chi) + \int S_t \chi \d\mu \right\}.                
    \end{align*}
    Since $f(t,\cdot)$ is Lipschitz continuous on $\mcl P_1(\R_+)$ and $\mcl P_\infty(\R_+)$ is dense in $\mcl P_1(\R_+)$, by density we can extend the equality of the last display to any $\mu \in \mcl P_1(\R_+)$.
\end{proof}

\section{Extension of concave functions} \label{s.extension}

As discussed above, the main difficulty to generalize the results obtained in Section \ref{s. scalar models} to the case $D > 1$ is the fact that the set $\mcl P^\upa(S^D_+)$ is nonconvex when $D > 1$. Indeed, this prevents us from performing the sup-inf interchange using \cite{sion1958minimax} as in the proof of Theorem~\ref{t.main scalar}. To circumvent this difficulty, we show that the initial condition $\psi : \mcl P^\upa(S^D_+) \to \R$ is the restriction of a concave and Lipschitz function defined on $\mcl P(S^D_+)$. Using this extension, we show that at $\mu = \delta_0$ the supremum over $\mcl P^\upa(S^D_+)$ in \eqref{e.lim.FN} can be rewritten as a supremum over $\mcl P(S^D_+)$. Since the set $\mcl P(S^D_+)$ is convex, we can then perform the sup-inf interchange and proceed as in the proof of Theorem~\ref{t.main scalar}. To do this properly, we will need to use some compacity properties, so in what follows $S^D_+$ is going to be replaced by a compact subset $\mcl K \subset S^D_+$.

\subsection{Optimal transport of monotone probability measures} \label{ss.general_concave_extension}

Let $\mcl K \subset S^D_+$ be a compact set and $\mcl P(\mcl K)$ denote the set of Borel probability measures on $\mcl K$. In this subsection, we study functions on possibly nonconvex subsets $K$ of $\mcl P(\mcl K)$. We prove that if such a function is concave on every convex subset of $K$, then it satisfies a Jensen-type inequality, provided that $K$ is an extreme set in its convex hull (see Definition~\ref{d.extreme}).

As previously, $\mcl L$ denotes the set of Lipschitz functions $\chi : S^D_+ \to \R$, $\mcl L_{\leq 1}$ denote the set of $1$-Lipschitz functions $\chi : S^D_+ \to \R$, $\mcl L^0$ denotes the set of functions $\chi \in \mcl L$ satisfying $\chi(0)= 0$ and $\mcl L_{\leq 1}^0 = \mcl L^0 \cap \mcl L_{\leq 1}$. We equip $\mcl P(\mcl K)$ with the optimal transport distance, for every $\mu,\nu \in \mcl P(\mcl K)$,
\begin{e*}
    d(\mu,\nu) = \inf_{\pi \in \Pi(\mu,\nu)} \left\{ \int_{\mcl K \times \mcl K} |x-y| \d \pi(x,y) \right\},
\end{e*}
where $\Pi(\mu,\nu)$ denotes the set of probability measures $\pi \in \mcl P(\mcl K \times \mcl K)$ such that $\pi_1 = \mu$ and $\pi_2 = \nu$. The distance $d$ also admits the following dual representation \cite[Theorem~5.10]{villani2},
\begin{equation} 
    d(\mu,\nu) = \sup_{\chi \in \mcl L_{\leq 1}} \left\{ \int_{S^D_+} \chi \d(\mu-\nu) \right\}.
\end{equation}
In what follows, we will consider a compact set $K \subset \mcl P(\mcl K)$. The set $\mcl P(K)$ of probability measures on $K$ will appear, we will equip this set with the optimal transport distance $\msf d$ derived from the optimal transport distance $d$ on $K$, that is for every $\eta,\eta' \in \mcl P(K)$,
\begin{e*}
    \msf d(\eta,\eta') = \inf_{\pi \in \Pi(\eta',\eta')} \left\{ \int_{ K \times K} d(\nu,\nu') \d \pi(\nu,\nu') \right\},
\end{e*}
where $\Pi(\eta,\eta')$ denotes the set of probability measures $\pi \in \mcl P(K \times K)$ such that $\pi_1 = \eta$ and $\pi_2 = \eta'$. As usual, the optimal transport distance $\msf d$ admits the dual representation \cite[Theorem~5.10]{villani2},
\begin{equation} 
    \msf d(\eta,\eta') = \sup_{X} \left\{ \int X(\nu) \d\eta(\nu) -\int X(\nu) \d\eta'(\nu) \right\},
\end{equation}
where the supremum is taken over the set of 1-Lipschitz functions $X : K \to \R$. 
\begin{proposition} \label{p.approx_by_dirac} 
    Let $E$ be a compact Polish space, we equip the set $\mcl P(E)$ of Borel probability measures on $E$ with the optimal transport distance. Let $\mu \in \mcl P(E)$, let $(X_i)_{i \geq 1} : \Omega \to E^\N$ be independent and identically distributed random variables with law $\mu$, then almost surely
    \begin{equation*}
        \frac{1}{n} \sum_{i=1}^n \delta_{X_i} \underset{n \to \infty}{\longrightarrow} \mu.
    \end{equation*}
    %
    %In particular, for almost all $\omega \in \Omega$, the sequence of deterministic finitely supported measures $(\frac{1}{n} \sum_{i=1}^n \delta_{X_i(\omega)})_n$ converges to $\mu$.
\end{proposition}

\begin{remark}
    By Caratheodory's extension theorem, given a Borel probability measure $\mu$ on a Polish space $E$, we can always construct a sequence of independent and identically distributed random variables with law $\mu$.     
    % Indeed, let $\Omega = E^\mathbb{N}$, $\mcl F$ the $\sigma$-algebra on $\Omega$ generated by cylinder sets on $\mcl B(E)$. By Caratheodory's extension theorem, there exists a unique probability measure $\P$  on $(\Omega, \mcl F)$ satisfying 
    % %
    % \begin{e*}
    %     \mathbb{P}(A_1 \times \dots \times A_k) = \mu(A_1) \times \dots \times \mu(A_k).
    % \end{e*}
    % %
    % If we set $X_i((\omega_j)_{j \geq 1}) = \omega_i$, then the sequence $(X_i)_{i \geq 1}$ is a sequence of random variables on $(\Omega,\mcl F, \P)$ that are independent and identically distributed with law $\mu$.
\end{remark}

\begin{proof}
    Let $d$ denote the optimal transport distance on $\mcl P(E)$. According to \cite[Theorem~6.9]{villani2}, since the space $E$ is compact, given a sequence $(\mu_n)_{n \geq 1}$ in $\mcl P(E)$, we have $\mu_n \to \mu$ in $(\mcl P(E),d)$ if and only if for every continuous function $h : E \to \R$,
    \begin{e*}
        \lim_{n \to +\infty } \int h \d\mu_n = \int h \d \mu.
    \end{e*}
    Let $\mcl C(E)$ denote the set of real-valued continuous functions on $E$. Since $E$ is compact and metrizable, $(\mcl C(E),|\cdot|_\infty)$ is separable \cite[Theorem~6.6]{conway1990fa}. Let $(h_p)_p$ be a dense sequence in $(\mcl C(E),|\cdot|_\infty)$, to check convergence in $(\mcl P(E),d)$, it is sufficient to check convergence against $h_p$ for every $p \geq 1$. By the law of large numbers, for every $p \geq 1$ the following holds almost surely
    \begin{equation*}
        \lim_{n \to \infty} \frac{1}{n} \sum_{i = 1}^n h_p(X_i) = \int_E h_p \d\mu.
    \end{equation*}
    Let $\tilde \mu_n = \frac{1}{n} \sum_{i = 1}^n \delta_{X_i} \in \mcl P(E)$, since $(h_p)_p$ is countable, almost surely we have for every $p \geq 1$,
    \begin{equation*}
        \lim_{n \to \infty} \int_E h_p \d\tilde \mu_n = \int_E h_p \d\mu.
    \end{equation*}
    That is, almost surely the sequence $(\tilde \mu_n)_n$ converges to $\mu$ in $(\mcl P(E),d)$. 
\end{proof}

Given $\eta \in \mcl P(K)$, we define the barycenter $\text{Bar}(\eta) \in \mcl P(\mcl K)$ by
\begin{e*}
    \text{Bar}(\eta)(A) = \int_{K} \nu(A) \d \eta(\nu).
\end{e*}
The barycenter of $\eta$ is the unique probability measure on $\mcl K$ satisfying for every continuous function $h : \mcl K \to \R$,
\begin{e*}
    \int_K \int_{\mcl K} h(x) \d\nu(x) \d\eta(\nu) = \int_{\mcl K} h(x) \d \text{Bar}(\eta)(x).
\end{e*}
Heuristically, we can think of $\text{Bar}(\eta)$ as the first moment of $\eta$.
\begin{remark} \label{r.finite_bar}
    When $\eta$ is a finitely supported probability measure, say of the form $\sum_{i = 1}^p c_i\delta_{\nu_i}$, we have $\text{Bar}(\eta) = \sum_{i = 1}^p c_i \nu_i$. 
\end{remark}
\begin{proposition} \label{p.def_bar}
    The barycenter map $\text{Bar} : (\mcl P(K), \msf d) \to (\mcl P(\mcl K),d)$ is Lipschitz continuous and for every $\eta,\eta' \in \mcl P(K)$ and $\lambda \in [0,1]$, we have 
    \begin{e*}
        \text{Bar}(\lambda \eta +(1-\lambda) \eta') = \lambda  \text{Bar}(\eta) + (1-\lambda)  \text{Bar}(\eta').
    \end{e*}
\end{proposition}
\begin{proof}
    Let $\chi \in \mcl L_{\leq 1}$, the map $\nu \mapsto \int \chi \d\nu$ is $1$-Lipschitz on $K$. Let $\eta,\eta' \in \mcl P(K)$, we have
    \begin{align*}
        d(\text{Bar}(\eta),\text{Bar}(\eta')) &= \sup_{\chi \in \mcl L_{\leq 1}} \left\{ \int_{\mcl K} \chi \d \text{Bar}(\eta) - \int_{\mcl K} \chi \d \text{Bar}(\eta') \right\} \\
                                              &= \sup_{\chi \in \mcl L_{\leq 1}} \left\{ \int_K \int_{\mcl K} \chi \d \nu \d\eta(\nu) - \int_K \int_{\mcl K} \chi \d \nu \d \eta'(\nu) \right\} \\ 
                                              &\leq \sup_X \left\{ \int_K X d\eta - \int_K X \d \eta' \right\} \\
                                              &= \msf d(\eta,\eta').
    \end{align*}
\end{proof}
Let $C$ denote the closed convex hull of $K$ in $\mcl P( \mcl K)$.
\begin{definition} \label{d.extreme}
    We say that $K$ is an extreme set in $C$ when for every $\eta \in \mcl P(C)$, if $\text{Bar}(\eta) \in K$ then $\eta \in \mcl P(K)$.
\end{definition}
For example, if $K$ is an extreme set in $C$, then taking $\eta = \lambda \delta_{\mu} + (1-\lambda) \delta_{\nu}$, where $\mu,\nu \in C$ and $\lambda \in (0,1)$, we have that 
\begin{e*}
    \lambda {\mu} + (1-\lambda) {\nu} \in  K \implies \mu,\nu \in K.
\end{e*}
\begin{proposition} \label{p.convex_hull=bar}
    The closed convex hull $C$ of $K$ satisfies 
    \begin{equation} \label{e.convex_hul=bar}
         C = \left\{ \text{Bar}(\eta) \big| \, \eta \in \mcl P(K) \right\}.
    \end{equation}
\end{proposition}
\begin{proof}
    According to Proposition~\ref{p.def_bar}, the set defined in \eqref{e.convex_hul=bar} is closed and convex. Let $C' \subset \mcl P(\mcl K)$ be a closed convex set containing $K$, let $\eta \in \mcl P(K)$. According to Proposition~\ref{p.approx_by_dirac}, there exists a sequence of finitely supported measures $(\eta_n)_n$ such that $\eta_n \to \eta$ as $n \to +\infty$. Since $C'$ is convex, according to Remark~\ref{r.finite_bar}, we have $\text{Bar}(\eta_n) \in C'$. In addition, since $C'$ is closed, letting $n \to \infty$ and using Proposition~\ref{p.def_bar} we have $\text{Bar}(\eta) \in C'$. We have proven that $\left\{ \text{Bar}(\eta)  \big| \, \eta \in \mcl P(K) \right\} \subset C'$, therefore $\left\{ \text{Bar}(\eta)  \big| \, \eta \in \mcl P(K) \right\}$ is the closed convex hull of $K$.
\end{proof}
\begin{proposition} \label{p.approximation_barycenter}
    Assume that $K$ is an extreme set in its convex hull. Let $\eta \in \mcl P(K)$ be such that $\text{Bar}(\eta) \in K$. There exists a sequence $\eta_n \in \mcl P(K)$ of finitely supported probability measures such that $\text{Bar}(\eta_n) \in K$ and $\eta_n \to \eta$ as $n \to +\infty$.
\end{proposition}
\begin{proof}
    Let $(\nu_i)_{i \geq 1}$ be a sequence of independent and identically distributed random measures in $K$ with law $\eta$, define 
    \begin{e*}
        \tilde \eta_n = \frac{1}{n} \sum_{i=1}^n \delta_{\nu_i}.
    \end{e*}
    According to Proposition~\ref{p.approx_by_dirac}, we have $\tilde \eta_n \to \eta$ almost surely. Let us show that almost surely for every $n \geq 1$, $\text{Bar}(\tilde \eta_n) \in K$. Define, 
    \begin{equation*}
        \mcl O = \{ \omega \in \Omega \big| \, \forall n \geq 1, \; \text{Bar}(\tilde \eta_n(\omega)) \in K \}.
    \end{equation*}
    Let $P_n \in \mcl P(C)$ denote the law of the random probability measure $\text{Bar}(\tilde \eta_n) \in C$. The probability measure $P_n$ is the only probability measure in $\mcl P(C)$ that satisfies for every $h : C \to \R$,
    \begin{equation*}
        \int_C h(\nu) \d P_n(\nu) = \int_{K^n} h \left( \frac{1}{n} \sum_{i=1}^n \nu_i\right) \prod_{i = 1}^n \d\eta(\nu_i).
    \end{equation*}
    For every $A \subset C$, we have,
    \begin{align*}
        \text{Bar}(P_n)(A) &= \int_C \nu(A) \d P_n(\nu) \\
                            &= \int_{K^n} \frac{1}{n} \sum_{i = 1}^n \nu_i(A) \prod_{i=1}^n \d\eta(\nu_i) \\
                            &= \int_K \nu(A) \d\eta(\nu) \\
                            &= \text{Bar}(\eta)(A).
    \end{align*}
    So $\text{Bar}(P_n) = \text{Bar}(\eta) \in K$, since $K$ is an extreme set in its closed convex hull, we have $P_n \in \mcl P(K)$ and almost surely $\text{Bar}(\tilde \eta_n) \in K$. Thus $\P(\mcl O) = 1$ and for almost all $\omega \in \Omega$, the measures $\eta_n = \tilde \eta_n (\omega)$ satisfy the desired properties.
\end{proof} 

\begin{proposition} \label{p.fake_jensen}
Assume that $K$ is an extreme set in its closed convex hull, let $\varphi : K \to \R$ be a bounded upper semicontinuous function. Assume that for every $\nu,\nu' \in K$, $\lambda \in [0,1]$ such that $\lambda \nu +(1-\lambda) \nu' \in K$ we have
\begin{equation*}
    \lambda \varphi(\nu) + (1-\lambda) \varphi(\nu') \leq \varphi \left( \lambda \nu +(1-\lambda) \nu' \right).
\end{equation*}
Then $\varphi$ satisfies Jensen's inequality, that is for every $\eta \in \mcl P(K)$ such that $\text{Bar}(\eta) \in K$, we have
\begin{equation} \label{e.fake_jensen}
    \int_K \varphi(\nu) \d \eta(\nu) \leq \varphi \left( \text{Bar}(\eta) \right).
\end{equation}
\end{proposition}

\begin{proof}
    Assume that $\eta$ is finitely supported, then $\eta = \sum_{i = 1}^p c_i \delta_{\nu_i}$ for some $\nu_i \in K$ and $c_i > 0$ such that $\sum_{i = 1}^p c_i = 1$ and $\sum_{i =1}^p c_i \nu_i \in K$. Since $K$ is an extreme set in its convex hull $C$, for every nonempty $I \subset \{1,\dots,p\}$, we have 
    \begin{e}
        \frac{\sum_{i \in I} c_i \nu_i}{\sum_{i \in I} c_i} \in K.
    \end{e}
    It follows by induction on $p \geq 1$, that 
    \begin{equation*}
        \sum_{i = 1}^p c_i \varphi(\nu_i) \leq \varphi \left(\sum_{i=1}^p c_i \nu_i\right).
    \end{equation*}
    Therefore, \eqref{e.fake_jensen} holds when $\eta$ is finitely supported. Assume now that $\eta \in \mcl P(K)$, is any probability measure such that $\text{Bar}(\eta) \in K$. According to Proposition~\ref{p.approximation_barycenter}, the probability measure $\eta$ can be approximated, by finitely supported probability measures $\eta_n \in \mcl P(K)$ with barycenters in $K$. Using Proposition~\ref{p.def_bar}, we can pass to the limit in
    \begin{equation*} 
    \int_K \varphi(\nu) d \eta_n(\nu) \leq \varphi \left( \text{Bar}(\eta_n) \right),
    \end{equation*}
    and use the upper semicontinuity of $\varphi$ to obtain \eqref{e.fake_jensen}.
 \end{proof}
\subsection{Monotone probability measures on a compact set}

Let $\mcl K$ be a compact subset of $S^D_+$, define 
\begin{e}
    \mcl P^\upa(\mcl K) = \mcl P^\upa(S^D_+) \cap \mcl P(\mcl K).
\end{e}
 In this section, we show that $K = \mcl P^\upa(\mcl K)$ is a compact set and that $K$ is extreme in its convex hull. This means that Proposition~\ref{p.fake_jensen} can be applied to functions on $\mcl P^\upa(\mcl K)$.
 
 Recall that we have defined a partial order on $S^D$ by setting $x \leq x'$ whenever $x' - x \in S^D_+$. We say that a subset $S \subset S^D_+$ is totally ordered when for every $x,x' \in S$ we have $x \leq x'$ or $x' \leq x$.
\begin{proposition} \label{p.totally_ordered_supp}
    The set $\mcl P^\upa(\mcl K)$ is a closed subset of $(\mcl P(\mcl K),d)$, in particular $\mcl P^\upa(\mcl K)$ is compact with respect to the topology induced by $d$. In addition, for every $\mu \in \mcl P(\mcl K)$, the support of $\mu$ is totally ordered if and only if $\mu \in \mcl P^\upa(\mcl K)$. 
\end{proposition}
\begin{proof}
    Let $\mu_n \in \mcl P^\upa(\mcl K)$ be a sequence that converges to some $\mu \in \mcl P(\mcl K)$. The sequence $(\msf q_{\mu_n})_n$ is a Cauchy sequence in $L^1([0,1),\mcl K)$, let $\msf q \in L^1([0,1),\mcl K)$ denote its limit. There exists a subsequence $(n_k)_k$ such that $\msf q_{\mu_{n_k}} \to \msf q$ almost everywhere, so there exists an almost everywhere representative of $\msf q$ in $\mcl Q_1(S^D_+)$, we also denote this representant by $\msf q$ and we have $\mu = \text{Law}(\msf q(U)) \in \mcl P^\upa(\mcl K)$. 

    Let $\mu \in \mcl P^\upa(\mcl K)$ then for every $x,y \in \text{supp}(\mu)$ there exists $u,v \in [0,1)$ such that $x = \msf q_\mu(u)$ and $y = \msf q_\mu(v)$. We must have $u \leq v$ or $v \leq u$, and since $\msf q_\mu$ is nondecreasing, this implies $x \leq y$ or $y \leq x$ and $\text{supp}(\mu)$ is totally ordered. Conversely, assume that $\text{supp}(\mu)$ is totally ordered, let $(X_i)_i$ be a sequence of independent and identically distributed random variables with law $\mu$, set 
    \begin{equation*}
        \tilde \mu_n = \frac{1}{n} \sum_{i=1}^n \delta_{X_i}.
    \end{equation*}
    According to Proposition~\ref{p.approx_by_dirac}, $\tilde \mu_n \to \mu$ almost surely. Let us fix $\omega \in \Omega$ such that $\tilde \mu_n(\omega) \to \mu$ and define $\mu_n = \tilde \mu_n(\omega)$. Since $\text{supp}(\mu)$ is totally ordered, for every $n \geq 1$ there exists a permutation $s_n \in \mcl S_n$ such that the sequence $(X_{s_n(i)}(\omega))_{1 \leq i \leq n}$ is nondecreasing. Let 
    \begin{equation*}
        \msf q_n = \sum_{i =1}^n X_{s_n(i)}(\omega)\mathbf{1}_{[\frac{i-1}{n},\frac{i}{n})},
    \end{equation*}
    we have $\msf q_n \in \mcl Q(S^D_+) \cap L^\infty$ and $\mu_n = \text{Law}(\msf q_n(U)) \in \mcl P^\upa(\mcl K)$. Since $\mu_n \to \mu$ and $\mcl P^\upa(\mcl K)$ is closed, we have $\mu \in \mcl P^\upa(\mcl K)$.
\end{proof}

\begin{proposition} \label{P.closed_convex_hull_of_monotone}
    The closed convex hull of $\mcl P^\upa(\mcl K)$ is $\mcl P(\mcl K)$ and $\mcl P^\upa(\mcl K)$ is an extreme set in $\mcl P(\mcl K)$.
\end{proposition}

\begin{proof}
    According to Proposition~\ref{p.convex_hull=bar}, to show that the closed convex hull of $\mcl P^\upa(\mcl K)$ is $\mcl P(\mcl K)$, it suffices to show that for every $\mu \in \mcl P(\mcl K)$, there exists $\eta \in  \mcl P( \mcl P^\upa(\mcl K))$ such that $\mu = \text{Bar}(\eta)$. 
    
    According to Proposition~\ref{p.approx_by_dirac}, there exists a sequence of finitely supported measures $(\mu_n)_n$ such that $\mu_n \to \mu$. For every $n \geq 1$, there exists $p_n, c_i^n,x_i^n$ such that $\mu_n = \sum_{i = 1}^{p_n}c_i^n \delta_{x_i^n}$. Define,
    \begin{equation*}
        \eta_n = \sum_{i=1}^{p_n} c_i^n \delta_{\delta_{x_i^n}},
    \end{equation*}
     we have $\mu_n = \text{Bar}(\eta_n)$. Up to extraction, the sequence $(\eta_n)_n$ converges to some  $\eta \in  \mcl P( \mcl P^\upa(\mcl K))$. Therefore, according to Proposition~\ref{p.def_bar}, we have $\mu = \text{Bar}(\eta)$. This proves that the closed convex hull of $\mcl P^\upa(\mcl K)$ is $\mcl P(\mcl K)$.

    Let $\eta$ be a probability measure on $\mcl P(\mcl K)$, such that $\text{Bar}(\eta) \in \mcl P^\upa(\mcl K)$. Let $A$ denote the support of $\text{Bar}(\eta)$, we have,
    \begin{equation*}
        0 = \text{Bar}(\eta)(A^c) = \int_{\mcl P(\mcl K)} \nu(A^c) d\eta(\nu).
    \end{equation*}
    So $\eta$-almost surely $\nu(A^c) = 0$. The set $A$ is a closed and according to Proposition~\ref{p.totally_ordered_supp}, $A$ is totally ordered. In particular, $\eta$-almost surely, the support of the measure $\nu$ is contained in $A$, therefore $\nu \in \mcl P^\upa(\mcl K)$. Thus, $\eta \in \mcl P( \mcl P^\upa(\mcl K))$, which concludes the proof.
\end{proof}

\subsection{Building the extension} 

In this section, we assume that the function $\xi : \R^{D \times D} \to \R$ is strictly convex on $S^D_+$. As usual $U$ denotes a uniform random variable in $[0,1)$, let $L^\infty_{\leq 1}$ denote the set of functions $[0,1) \to S^D_+$ that are essentially bounded by $1$ in norm and define
\begin{equation}
    \mcl P_\xi^\upa = \left\{ \text{Law}(\nabla \xi(\msf q(U))) \big|\,  \msf q \in \mcl Q(S^D_+) \cap L^\infty_{\leq 1} \right\} \subset \mcl P^\upa(S^D_+).
\end{equation}
Recall that $B(0,1)$ denotes the unit ball centered at $0$ in $S^D$, we let 
\begin{e*}
    \mcl K_\xi = \nabla \xi(B(0,1) \cap S^D_+).
\end{e*} 
The goal of this subsection is to prove that any Lipschitz and concave function on the nonconvex set $\mcl P_\xi^\upa$ extends into a Lipschitz and concave function on $\mcl M_1(\mcl K_\xi)$ with the same Lipschitz constant.

Note that we have $\mcl P_\xi^\upa \subset \mcl P^\upa(\mcl K_\xi)$, but in general the inclusion is strict when $D > 1$. The set $\mcl P_\xi^\upa$ is the image of $\mcl P^\upa(B(0,1) \cap S^D_+)$ by the mapping $\mu \mapsto \nabla \xi_\# \mu$. The closed convex envelope and the closed linear span of $\mcl P_\xi^\upa$ are respectively $\mcl P(\mcl K_\xi)$ and $\mcl M_1(\mcl K_\xi)$. Note that since $\mcl K_\xi$ is compact, $\mcl M_1(\mcl K_\xi)$ is in fact equal to the set of signed measures on $\mcl K_\xi$ with finite mass. 
\begin{definition}
    We say that a function $\varphi : \mcl P_\xi^\upa \to \R$ is pre-concave, when for every $\mu,\nu \in \mcl P^\upa_\xi$ and $\lambda \in [0,1]$,
    \begin{equation*}
        \lambda \mu + (1-\lambda) \nu \in \mcl P^\upa_\xi \implies \lambda \varphi( \mu) + (1-\lambda) \varphi(\nu) \leq \varphi( \lambda \mu + (1-\lambda) \nu). 
    \end{equation*}
\end{definition}
As in Section~\ref{s.fenchel-moreau}, we equip $\mcl M_1(\mcl K_\xi)$ with the norm $|\cdot|_{\mcl M}$ and we denote by $d$ the distance induced by $|\cdot|_{\mcl M}$. The restriction of the distance $d$ to $\mcl P(\mcl K_\xi)$ is the optimal transport distance. Given a continuous function $\varphi : \mcl P_\xi^\upa \to \R$, for every $\mu' \in \mcl M_1( \mcl K_\xi)$, we define
\begin{equation} \label{e.def ext}
     \bar \varphi(\mu') = \sup_{\eta \in \mcl P \left( \mcl P^\upa_\xi \right)} \left\{ \int \varphi d\eta - d(\text{Bar}(\eta), \mu') \right\}.
\end{equation}
For every function $\chi \in \mcl L(\mcl K_\xi)$, we also define 
\begin{equation}
    \varphi^\xi_*(\chi) = \inf_{\mu \in \mcl P_\xi^\upa} \left\{ \int \chi d\mu - \varphi(\mu) \right\}. 
\end{equation}
We choose $x_0 = \nabla \xi(0)$ as a reference point for the norm $|\cdot|_{\mcl L}$ in $\mcl L(\mcl K_\xi)$.
we equip $\mcl P(\mcl P^\upa_\xi)$ with the optimal transport distance inherited from the restriction of the distance $d$ to $\mcl P^\upa_\xi$. That is for every $\eta,\eta' \in \mcl P(\mcl P^\upa_\xi)$,
\begin{equation}
    \msf d(\eta,\eta') = \inf_{\pi \in \Pi(\eta,\eta')} \left\{ \int_{\mcl P^\upa_\xi} d(\nu,\nu') \d\pi(\nu,\nu') \right\}.
\end{equation}
Here $\Pi(\eta,\eta')$ denotes the set of measures $\pi \in \mcl P (\mcl P^\upa_\xi \times \mcl P^\upa_\xi)$ with marginals $\pi_1 = \eta$ and $\pi_2 = \eta'$. As usual, the optimal transport distance $\msf d$ admits the dual representation \cite[Theorem~5.10]{villani2},
\begin{equation} \label{e.double OT distance dual}
    \msf d(\eta,\eta') = \sup_{X} \left\{ \int X(\nu) \d\eta(\nu) -\int X(\nu) \d\eta'(\nu) \right\},
\end{equation}
where the supremum is taken over the set of 1-Lipschitz functions $X : \mcl P^\upa_\xi \to \R$. 

The goal of this subsection is to show the following extension theorem, which will be applied to the function $\psi$ defined in \eqref{e.def_psi} in the next section.
\begin{theorem} \label{t.extension} 
    For every $1$-Lipschitz and pre-concave function $\varphi : \mcl P_\xi^\upa \to \R$, the function $\bar \varphi : \mcl M_1(\mcl K_\xi) \to \R$ defined by \eqref{e.def ext} is a $1$-Lipschitz and concave extension of $\varphi$. In addition, for every $\mu' \in \mcl M(\mcl K_\xi)$, we have 
    \begin{equation} \label{e.bidual ext}
        \overline \varphi(\mu') = \min_{\chi \in \mcl L_{ \leq 1}(\mcl K_\xi)} \left\{ \int \chi \d\mu - \varphi^\xi_*(\chi) \right\}.
    \end{equation}
\end{theorem}

Note that given $S \subset S^D_+$ and $\chi : S \to \R$, by letting 
\begin{e*}
    \overline{\chi}(x) = \inf_{y \in S } \left\{ \chi(y) + |\chi|_{\textrm{Lip}}|x-y|\right\},
\end{e*}
we define a Lipschitz extension of $\chi$ to $S^D_+$ with the same Lipschitz constant. Therefore, in Theorem~\ref{t.extension}, \eqref{e.bidual ext} remains true if we take the minimum over $\mcl L_{ \leq 1}(S^D_+)$ rather than $\mcl L_{ \leq 1}(\mcl K_\xi)$.

\begin{proposition} \label{p.injective_gradient}
    Let $\xi \in \mcl C^\infty(\R^{D \times D},\R)$ be a strictly convex function on $S^D_+$, then the function $\nabla \xi : S^D_+ \to \R^{D \times D}$ is injective.
\end{proposition}

\begin{proof}
    Let $a \neq b$ in $S^D_+$, the function $\gamma : t \mapsto \xi(ta +(1-t)b)$ is strictly convex on $[0,1]$, so its derivative is strictly increasing on $(0,1)$ and for every $s < t$ we have
    \begin{equation}
       \nabla \xi(sa +(1-s)b) \cdot(a-b) < \nabla \xi(ta +(1-t)b) \cdot (a-b).
    \end{equation}
    Letting $s \to 0$ and $t \to 1$, we obtain 
    \begin{equation}
        \xi(b) \cdot(a-b) \leq \gamma'(1/3) <\gamma'(2/3) \leq \xi(a) \cdot (a-b).
    \end{equation}
    This justifies $\nabla \xi(a) \neq \nabla \xi(b)$.
\end{proof}

\begin{proposition} \label{p.Pupaxi is an extreme set }
    The set $\mcl P^\upa_\xi$ is compact and is an extreme set in its closed convex hull $\mcl P(\mcl K_\xi)$.
\end{proposition}

\begin{proof}
    The set $\mcl P^\upa_\xi$ is the image of $\mcl P^\upa(B(0,1) \cap S^D_+)$ by the continuous map $ \nu \mapsto \nabla \xi_\# \nu$. According to Proposition \ref{p.totally_ordered_supp}, $\mcl P^\upa(B(0,1) \cap S^D_+)$ is compact, so $\mcl P^\upa_\xi$ is compact. Let $\eta \in \mcl P(\mcl P(\mcl K_\xi))$ be such that $\text{Bar}(\eta) \in \mcl P^\upa_\xi$ and let $\rho \in \mcl P^\upa(B(0,1) \cap S^D_+)$ such that $\text{Bar}(\eta) = \nabla \xi_\# \rho$. In particular, we have $\text{Bar}(\eta) \in \mcl P^\upa( \mcl K_\xi )$, since $\mcl P^\upa( \mcl K_\xi)$ is an extreme set in its convex hull according to Proposition \ref{P.closed_convex_hull_of_monotone}, it follows that $\eta$ is supported on $\mcl P^\upa( \mcl K_\xi)$. Therefore, for every $\mu$ in the support of $\eta$, we have $\mu \in \mcl P^\upa(\mcl K_\xi)$ so the path $\msf q_\mu \in \mcl Q_\infty(S^D_+)$ is well-defined and valued in $\mcl K_\xi$. In particular given $\mu$ in the support of $\eta$, there exists a, possibly non-monotone, map $\msf p_\mu : [0,1) \to B(0,1) \cap S^D_+$ such that for every $u \in [0,1)$,
    \begin{e*}
        \msf q_\mu(u) = \nabla \xi( \msf p_\mu(u)).
    \end{e*}
    Moreover, we have $\mu = \nabla \xi_\# \text{Law}(\msf p_\mu(U))$, this implies
    \begin{equation}
        \nabla \xi_\# \int \text{Law}(\msf p_\mu(U)) \d\eta(\mu) = \nabla \xi_\#  \rho.
    \end{equation}
    According to Proposition \ref{p.injective_gradient}, $\nabla \xi : S^D_+ \to S^D_+$ is injective, so we have 
    \begin{e*}
        \int \text{Law}(\msf p_\mu(U)) \d\eta(\mu) = \rho.
    \end{e*}
    Finally, since $\rho \in \mcl P^\upa(B(0,1) \cap S^D_+)$, and $\mcl P^\upa( B(0,1) \cap S^D_+)$ is an extreme set in its convex hull according to Proposition \ref{P.closed_convex_hull_of_monotone}, we deduce that $\eta$-almost surely $\text{Law}(\msf p_\mu(U)) \in \mcl P^\upa( B(0,1) \cap S^D_+)$. Hence, $\eta$-almost surely $\mu \in \mcl P^\upa_\xi$ which concludes.
    \end{proof}
\begin{proposition} \label{p.ext is lip and concave}
    For any bounded function $\varphi : \mcl P^\upa_\xi \to \R$, the function $\bar \varphi$ defined by \eqref{e.def ext} is concave and $1$-Lipschitz on $\mcl M_1(\mcl K_\xi)$.
\end{proposition}
\begin{proof}
    The function $(\eta,\mu') \mapsto \int \varphi \d\eta - d(\text{Bar}(\eta), \mu')$ is concave on $\mcl P \left( \mcl P^\upa_\xi \right) \times \mcl M_1(\mcl K_\xi)$, so $\bar \varphi$ is concave as the supremum of a jointly concave functionnal. For every $\eta \in \mcl P \left( \mcl P^\upa_\xi \right)$, the function $\mu' \mapsto \int \varphi \d\eta - d(\text{Bar}(\eta), \mu')$ is $1$-Lipschitz, so $\bar \varphi$ is $1$-Lipschitz as the pointwise supremum of a family of $1$-Lipschitz functions.
\end{proof}
\begin{proposition} \label{p.bar varphi is an ext}
    Assume that $\varphi : \mcl P^\upa_\xi \to \R$ is $1$-Lipschitz and pre-concave, then the function $\bar \varphi$ defined by \eqref{e.def ext} is an extension of $\varphi$.
 \end{proposition}
\begin{proof}
    We fix $\mu' \in \mcl P^\upa_\xi$ and our goal is to show that $\Bar{\varphi}(\mu') = \varphi(\mu')$. Choosing $\eta = \delta_{\mu'}$ in \eqref{e.def ext}, we obtain $\Bar{\varphi}(\mu') \geq \varphi(\mu')$, so we only need to show the other bound.

    \noindent Step 1. We show that for every $\eta \in \mcl P(\mcl P^\upa_\xi)$, there exists $\eta' \in \mcl P(\mcl P^\upa_\xi)$ such that $\text{Bar}(\eta') = \mu'$ and $\msf d(\eta,\eta') \leq d(\text{Bar}(\eta),\mu')$.

    \noindent Let $\pi$ be an optimal coupling between $\text{Bar}(\eta)$ and $\mu'$, the existence of such a $\pi$ is guaranteed by \cite[Theorem~4.1]{villani2}. According to \cite[Theorem~5.3.1]{AGS}, there exists a family of probability measures $(\mu_x)_{x \in \mcl K_\xi}$ such that,
    %
    %\viccomment{If needed, it is possible to modify the following argument using a deterministic optimal coupling of the form $\pi = (id,T)_\# \text{Bar}(\eta)$ when $\text{Bar}(\eta)$ is absolutely continuous with respect to the Lebesgues measure and otherwise proceeding by approximation of $\eta$.}
    \begin{equation*}
        \int_{\mcl K_\xi \times \mcl K_\xi} h(x,y) \d\pi(x,y) = \int_{\mcl K_\xi} \int_{\mcl K_\xi} h(x,y) \d\mu_x(y) \d\text{Bar}(\eta)(x).
    \end{equation*}
   Given a random probability measure $\nu \in \mcl P^\upa_\xi$ sampled from $\eta$, we define $\nu' \in \mcl P(\mcl K_\xi)$ as the unique probability measure satisfying 
   \begin{equation*}
      \int_{\mcl K_\xi} h(y) \d\nu'(y) = \int_{\mcl K_\xi} \int_{\mcl K_\xi} h(y) \d\mu_x(y) \d\nu(x).
   \end{equation*}
   We let $\eta' \in \mcl P(\mcl P(\mcl K_\xi))$ be the law of the random variable $\nu' \in \mcl P(\mcl K_\xi)$. We have,
   \begin{align*}
     \int_{\mcl K_\xi} h(y) \d \text{Bar}(\eta')(y) &= \int_{\mcl P^\upa_\xi} \int_{\mcl K_\xi}h(y) \d\nu'(y) \d \eta(\nu) \\
                                               &= \int_{\mcl P^\upa_\xi}  \int_{\mcl K_\xi} \int_{\mcl K_\xi} h(y) \d\mu_x(y) \d\nu(x) \d \eta(\nu) \\
                                               &= \int_{\mcl K_\xi} \int_{\mcl K_\xi} h(y) \d\mu_x(y) \d\text{Bar}(\eta)(x) \\
                                               &= \int_{\mcl K_\xi \times \mcl K_\xi} h(y) \d\pi(x,y) \\
                                               &=  \int_{\mcl K_\xi} h(y) \d\mu'(y).
   \end{align*}
    Therefore, $\text{Bar}(\eta') = \mu' \in \mcl P^\upa_\xi$. According to Proposition~\ref{p.Pupaxi is an extreme set }, this imposes that $\eta'$ is supported on $\mcl P^\upa_\xi$. Finally, by definition of the optimal transport distance, we have,
    \begin{align*}
        \msf d(\eta,\eta') &\leq \int_{\mcl P^\upa_\xi } \d(\nu,\nu') \d\eta(\nu) \\
                      &\leq \int_{\mcl P^\upa_\xi} \int_{\mcl K_\xi} \int_{\mcl K_\xi} \d(x,y) \d\mu_x(y) \d\nu(x) \d\eta(\nu) \\
                      &=    \int_{\mcl K_\xi} \int_{\mcl K_\xi} d(x,y) \d\mu_x(y) \d\text{Bar}(\eta)(x) \\
                      &=    \int_{\mcl K_\xi} \int_{\mcl K_\xi} d(x,y) \d\pi(x,y) \\
                      &=    d(\text{Bar}(\eta), \mu').
    \end{align*}
    This concludes Step 1.
    
    \noindent Step 2. We show that $\Bar{\varphi}(\mu') \leq \varphi(\mu')$.

    \noindent Let $\eta \in \mcl P(\mcl P^\upa_\xi)$ and let $\eta' \in \mcl P(\mcl P^\upa_\xi)$ be as built in Step 1. We have,
    \begin{equation*}
        \int \varphi \d \eta - d ( \text{Bar}(\eta), \mu') = \int \varphi \d \eta' + \int \varphi \d(\eta - \eta')  - d (\text{Bar}(\eta), \mu').
    \end{equation*}
    By Proposition \ref{p.fake_jensen}, $\int \varphi \d \eta' \leq \varphi(\mu')$. Since $\varphi$ is $1$-Lipschitz, according to \eqref{e.double OT distance dual}, we have
    \begin{e*}
        \int \varphi \d(\eta - \eta') \leq \msf d(\eta,\eta').
    \end{e*}
    In Step 1, we have built $\eta'$ so that $- d ( \text{Bar}(\eta), \mu') \leq -\msf d(\eta,\eta')$, therefore, it follows from the two previous displays that
    \begin{equation*}
        \int \varphi \d \eta - d ( \text{Bar}(\eta), \mu') \leq \varphi(\mu').
    \end{equation*}
    By taking the supremum over $\eta \in \mcl P(\mcl P^\upa_\xi)$ on the left-hand side, we can conclude.
\end{proof}

\begin{proposition} \label{p.dual of ext}
    Let $\varphi : \mcl P^\upa_\xi \to \R$, the function $\bar \varphi$ defined in \eqref{e.def ext} satisfies for every $\chi_0 \in \mcl L_{\leq 1}(\mcl K_\xi)$, $\bar \varphi_*(\chi_0) = \varphi^\xi_*(\chi_0)$.
\end{proposition}

\begin{proof}
    Observe that when $\chi_0 \in \mcl L_{\leq 1}(\mcl K_\xi)$, the map $\mu \mapsto \int \chi_0 \d\mu$ is $1$-Lipschitz on $\mcl M_1(\mcl K_\xi)$ with respect to the distance $d$. In particular, for every $\nu \in \mcl P^\upa_\xi$, we have  
    \begin{equation*}
        \inf_{\mu \in \mcl M_1(\mcl K_\xi)} \left\{ \int \chi_0 \d\mu + d(\nu,\mu) \right\} = \int \chi_0 \d\nu.
    \end{equation*}
    Thus,
    \begin{align*}
        \bar \varphi_*(\chi_0) &= \inf_{\mu \in \mcl M_1(\mcl K_\xi)} \left\{ \int \chi_0 \d\mu - \bar \varphi(\mu) \right\} \\
                               &= \inf_{\mu \in \mcl M_1(\mcl K_\xi)} \inf_{\eta \in \mcl P(\mcl P^\upa_\xi)} \left\{ \int \chi_0 \d\mu - \int \varphi \d \eta + d(\text{Bar}(\eta),\mu)\right\} \\
                               &= \inf_{\eta \in \mcl P(\mcl P^\upa_\xi)} \left\{- \int \varphi \d \eta + \inf_{\mu \in \mcl M_1(\mcl K_\xi)}  \left\{ \int \chi_0 \d\mu  + d(\text{Bar}(\eta),\mu)\right\} \right\}\\
                               &= \inf_{\eta \in \mcl P(\mcl P^\upa_\xi)} \left\{- \int \varphi \d \eta + \int \chi_0 \d \text{Bar}(\eta) \right\} \\
                               &= \inf_{\eta \in \mcl P(\mcl P^\upa_\xi)} \left\{\int \left( -\varphi(\nu) + \int \chi_0 \d\nu \right) \d \eta(\nu) \right\} \\
                               &= \inf_{\nu \in \mcl \mcl P^\upa_\xi} \left\{ -\varphi(\nu) + \int \chi_0 \d\nu \right\} \\
                               &= \varphi^\xi_*(\chi_0).                           
    \end{align*}    
\end{proof}

\begin{proof}[Proof of Theorem \ref{t.extension}]
    The function $\overline \varphi$ is $1$-Lipschitz and concave on the set\\ $\mcl M_1(\mcl K_\xi)$ according to Proposition~\ref{p.ext is lip and concave}. It follows from Corollary~\ref{c.fenchel moreau for lipschitz}, that for every $\mu' \in \mcl M_1(\mcl K_\xi)$, 
    \begin{e*}
        \overline{\varphi}(\mu') =  \min_{\chi \in \mcl L_{\leq 1}(\mcl K_\xi) } \left\{ \int \chi \d\mu - \overline{\varphi}_*(\chi) \right\}.
    \end{e*}
    According to Proposition~\ref{p.dual of ext}, we have for every $\chi \in \mcl L_{\leq 1}(\mcl K_\xi)$, $\overline{\varphi}_*(\chi) = {\varphi}^\xi_*(\chi)$, thus
    \begin{e*}
        \overline{\varphi}(\mu') =  \min_{\chi \in \mcl L_{\leq 1}(\mcl K_\xi) } \left\{ \int \chi \d\mu - {\varphi}^\xi_*(\chi) \right\}.
    \end{e*}
    Finally, we have shown in Proposition~\ref{p.bar varphi is an ext} that $\bar \varphi$ is an extension of $\varphi$.
\end{proof}

\section{Vector models}  \label{s.vector models}
We now adapt the arguments of the proof of Theorem~\ref{t.main scalar} to prove Theorem~\ref{t.main}. To do so we will use Theorem~\ref{t.extension}, which is proven to be valid only when $\xi$ is strictly convex on $S^D_+$. Therefore, the proof we give is only valid when $\xi$ is assumed to be strictly convex on $S^D_+$. We will fix this issue in  Section~\ref{s.convex} using a continuity argument. 
\subsection{Strictly convex models} \label{ss.strictly convex}

Let $\underline \xi = \xi + I_{B(0,1) \cap S^D_+}$ denote the function that coincides with $\xi$ on $B(0,1) \cap S^D_+$ and is equal to $+\infty$ outside of $B(0,1) \cap S^D_+$. We define
\begin{e*}
    \theta(x) = x \cdot \nabla \xi(x) - \xi(x).
\end{e*}
\begin{proposition} \label{p. expression theta}
    The function $\underline \xi ^* : S^D_+ \to \R$ is $1$-Lipschitz and for every $x \in B(0,1)\cap S^D_+$, we have 
    \begin{equation*}
        \theta(x) = \underline \xi^*(\nabla \xi(x)).
    \end{equation*}
\end{proposition}

\begin{proof} 
    For every $y \in \R^{D \times D}$, we have 
    \begin{equation*}
        \underline \xi ^*(y) = \sup_{ x \in B(0,1) \cap S^D_+} \left\{ x \cdot y - \xi(x) \right\}.
    \end{equation*}
    Therefore, $\underline \xi ^*$ is $1$-Lipschitz as the supremum of a family of $1$-Lipschitz functions. 
    
    Let $x \in B(0,1) \cap S^D_+$, from the definition of $\theta$, it is clear that $\theta(x) \leq \xi^*(\nabla \xi(x))$. Given $x' \in S^D_+$, by convexity of $\xi$ on $S^D_+$, we have for every $\lambda \in (0,1]$,
    \begin{e*}
        \frac{\xi(\lambda x' +(1-\lambda)x) - \xi(x)}{\lambda} \leq \xi(x') -\xi(x).
    \end{e*}
    Letting $\lambda \to 0$ in the previous display, we obtain 
    \begin{e*}
        x \cdot \nabla \xi(x) - \xi(x) \geq x' \cdot \nabla \xi(x) - \xi(x').
    \end{e*}
    This proves that,
    \begin{e*}
        \theta(x) = \sup_{x' \in S^D_+ } \left\{ x' \cdot \nabla \xi(x) - \xi(x') \right\} = \xi^*( \nabla \xi(x)).
    \end{e*}
    In particular, the supremum in the previous display is reached at $x' = x \in S^D_+ \cap B(0,1)$ and we have,
    \begin{equation*}
        \theta(x) =  \sup_{x' \in S^D_+ \cap B(0,1) } \left\{ x' \cdot \nabla \xi(x) - \xi(x') \right\} = \underline \xi ^*(\nabla \xi(x)). \qedhere
    \end{equation*}
\end{proof}

Recall that for every Lipschitz function $\chi : S^D_+ \to \R$, we have defined for every $x \in S^D_+$,
\begin{e*}
    \tilde S_t \chi(x) = \sup_{y \in S^D_+ \cap B(0,1)} \left\{ \chi(x + t\nabla \xi(y)) - t \xi^* \left(\nabla \xi(y) \right) \right\}.
\end{e*}
\begin{theorem} \label{t.main strictly convex}
    Assume that $\xi$ is strictly convex on $S^D_+$, then Theorem~\ref{t.main} holds.
\end{theorem}
\begin{proof}
    According to \cite[Corollary~8.7~(2)]{chenmourrat2023cavity}, the following refinement of \eqref{e.lim.FN} holds,
    \begin{e*}
        \lim_{N \to +\infty} \bar F_N(t) = \sup_{\mu \in \mcl P_\xi^\upa} \left\{ \psi(\mu)  - t \int \xi^* \left( \frac{\cdot}{t} \right) \d\mu\right\}.
    \end{e*}
    We have, $(t\xi)^* = t \xi^* \left( \frac{\cdot}{t}\right)$ so up to replacing $\xi$ by $t\xi$, we may assume without loss of generality that $t  = 1$.
    
    \noindent Step 1. We show that
    \begin{e*}
      \lim_{N \to +\infty} \bar F_N(1) = \sup_{\mu \in \mcl P(\mcl K_\xi)} \left\{ \bar \psi(\mu) - \int \underline \xi^* \d\mu \right\}.  
    \end{e*}

    \noindent We have
    \begin{equation*}
        \lim_{N \to +\infty} \bar F_N(1) = \sup_{ \mu \in \mcl P^\upa(B(0,1) \cap S^D_+)} \left\{ \psi( \nabla \xi_\# \mu ) - \int\theta \d\mu \right\}.
    \end{equation*}
    According Proposition~\ref{p. expression theta}, we have $\int\theta(x) \d\mu(x) = \int \underline \xi^*(\nabla \xi(x)) \d\mu(x)$, replacing $\mu$ by $\nabla \xi_\# \mu$, we obtain 
    \begin{equation*}
        \lim_{N \to +\infty} \bar F_N(1) = \sup_{ \mu \in \mcl P_\xi^\upa} \left\{ \psi( \mu ) - \int \underline \xi ^* \d\mu \right\}.
    \end{equation*}
    In addition, it follows from Theorem~\ref{t.extension} that $\bar \psi$ is an extension of $\psi$, therefore the previous display implies
    \begin{equation}
       \lim_{N \to +\infty} \bar F_N(1) \leq \sup_{\mu \in \mcl P(\mcl K_\xi)} \left\{ \bar \psi(\mu) - \int \underline \xi^* \d\mu \right\}.
    \end{equation}
    Conversely, let us fix $\mu \in \mcl P(\mcl K_\xi)$, for every $\varepsilon > 0$, there exists $\eta \in \mcl P(\mcl P^\upa_\xi)$, such that
    \begin{equation*}
        \bar \psi(\mu) \leq \varepsilon + \int \psi \d\eta -d(\text{Bar}(\eta),\mu).
    \end{equation*}
    Since $\underline \xi^*$ is $1$-Lipschitz according to Proposition~\ref{p. expression theta} by definition of the distance $d$, it follows that,
    \begin{align*}
        \bar \psi(\mu) - \int \underline \xi^* \d\mu &\leq \varepsilon + \int \psi \d\eta  - \int \underline \xi^* \d\mu -d(\text{Bar}(\eta),\mu) \\
        &\leq \varepsilon + \int \left( \psi(\nu) - \int \underline \xi^* \d\nu \right) \d\eta(\nu) 
        \\&\qquad  \qquad  \qquad  + \int \underline \xi^* \d( \text{Bar}(\eta) - \mu) - d(\text{Bar}(\eta),\mu) \\
        &\leq \varepsilon + \int \left( \psi(\nu) - \int \underline \xi^* \d\nu \right) \d\eta(\nu) \\
        &\leq \varepsilon + \sup_{\nu \in \mcl P_\xi^\upa} \left\{ \psi(\nu) - \int \underline \xi^* \d\nu \right\} \\
        &=\varepsilon + \lim_{N \to +\infty } \bar F_N(1).
    \end{align*}
    Taking the supremum over $\mu \in \mcl P(\mcl K_\xi)$, this yields
    \begin{equation*}
       \sup_{\mu \in \mcl P(\mcl K_\xi)} \left\{ \bar \psi(\mu) - \int \underline \xi^* \d\mu \right\} \leq \varepsilon + \lim_{N \to +\infty } \bar F_N(1).
    \end{equation*}
    Finally, since $\varepsilon > 0$ is arbitrary, we obtain the desired result by letting $\varepsilon \to 0$.

    \noindent Step 2. We show that 
    \begin{e*}
        \lim_{N \to +\infty} \bar F_N(1) = \inf_{\chi \in \mcl L^0_{\leq 1}} \sup_{\mu \in \mcl P(\mcl K_\xi)} \left\{ \int \chi \d\mu - \psi^\xi_*(\chi) - \int \underline \xi^* \d\mu \right\}.
    \end{e*}

    \noindent According to Theorem~\ref{t.extension}, we have for every $\mu \in \mcl P(\mcl K_\xi)$, 
    \begin{e*}
        \bar \psi(\mu) = \inf_{\chi \in \mcl L^0_{\leq 1}} \left\{\int \chi \d\mu - \psi^\xi_*(\chi) \right\}.
    \end{e*}
    It thus follows from Step 1 that
    \begin{equation*}
        \lim_{N \to +\infty} \bar F_N(1) =  \sup_{\mu \in \mcl P(\mcl K_\xi)} \inf_{\chi \in \mcl L^0_{\leq 1}}  \mcl G(\chi,\mu),
    \end{equation*}
    where,
    \begin{equation*}
        \mcl G(\chi,\mu) = \int \chi \d\mu - \psi^\xi_*(\chi) - \int \underline \xi^* \d\mu.
    \end{equation*}
    The set $\mcl P(\mcl K_\xi)$ is convex compact with respect to the topology of the distance $d$. For every $\chi \in \mcl L^0_{\leq 1}$ the map $\mcl G(\chi,\cdot)$ is concave and Lipschitz continuous with respect to $d$. For every $\mu \in \mcl P(\mcl K_\xi)$, the map $\mcl G(\cdot,\mu)$ is convex and lower semi continuous with respect to the topology of local uniform convergence. Therefore, according to \cite[Corollary~3.4]{sion1958minimax} we can perform a sup-inf interchange in the previous display to obtain,
    \begin{equation*}
        \lim_{N \to +\infty} \bar F_N(1) =  \inf_{\chi \in \mcl L^0_{\leq 1}} \sup_{\mu \in \mcl P(\mcl K_\xi)} \mcl G(\chi,\mu).
    \end{equation*}
    \noindent Step 3. Conclusion.
    
    \noindent Using Step 2, we have 
    \begin{align*}
        \lim_{N \to +\infty} \bar F_N(1) &= \inf_{\chi \in \mcl L^0_{\leq 1}} \left\{ - \psi^\xi_*(\chi) + \sup_{\mu \in \mcl P(\mcl K_\xi)} \left\{ \int \chi d\mu - \int \underline \xi^* \d\mu \right\} \right\}\\  
                    &= \inf_{\chi \in \mcl L^0_{\leq 1}} \left\{ - \psi^\xi_*(\chi) + \sup_{x \in B(0,1) \cap S^D_+} \left\{  \chi(\nabla \xi(x)) - \underline \xi^*(\nabla \xi(x)) \right\} \right\}\\ 
                    &= \inf_{\chi \in \mcl L^0_{\leq 1}} \left\{ \tilde S_t \chi(0) - \psi^\xi_*(\chi) \right\}.
        \end{align*}
\end{proof}

\subsection{Convex models} \label{s.convex}

In this subsection, we prove Theorem~\ref{t.main} when $\xi$ is only assumed to be convex but not necessarily strictly convex. To do so, we will use the fact that for every $\alpha > 0$, $x \mapsto \xi(x) + \alpha |x|^2$ is strictly convex. This allows us to establish \eqref{e.main} with $\alpha > 0$ using Theorem~\ref{t.main strictly convex}, and we can then let $\alpha \to 0$ to conclude. 

We start by defining a family of Gaussian processes $(H^\alpha_N)_{N \geq 1}$ with covariance function $\xi_\alpha(x) = \xi(x) + \alpha |x|^2$. For every $\sigma \in (\R^D)^N$, we let
\begin{e*}
    H_N^\text{Potts}(\sigma)  = \frac{1}{\sqrt{N}} \sum_{i,j = 1}^N J_{ij} \sigma_i \cdot \sigma_j,
\end{e*}
where the $J_{ij}$'s are independent standard Gaussian random variables. We choose $(J_{ij})_{i,j \geq 1}$ independent of $(H_N(\sigma))_{\sigma \in \R^{D \times N}}$. For every $\alpha \geq 0$, let
\begin{e*}
    H_N^\alpha(\sigma) = H_N(\sigma) + \sqrt{\alpha} H_N^{\text{Potts}}(\sigma).
\end{e*}
We have for every  $\sigma,\tau \in (\R^D)^N$,
\begin{e*}
    \E \left[ H_N^\alpha(\sigma) H_N^\alpha(\tau)\right] = N \xi_\alpha \left(\frac{\sigma \tau^\perp}{N} \right).
\end{e*}
We let $\bar F^\alpha_N(t)$ denote the free energy of $H_N^\alpha$, more precisely
\begin{e*}
    \bar F^\alpha_N(t) = -\frac{1}{N} \E \log \int \exp \left( \sqrt{2t}H^\alpha_N(\sigma) - Nt \xi_\alpha \left( \frac{\sigma \sigma^\perp}{N}\right)\right) \d P_N(\sigma).
\end{e*}
Note that at $\alpha = 0$, we have $\xi_0 = \xi$, $H_N^0 = H_N$ and $\bar F^0_N(t) = \bar F_N(t)$. We let $\langle \cdot \rangle_\alpha$ denote the Gibbs measure associated to $\bar F^\alpha_N(t)$, it is a random probability measure on $\R^{D \times N}$ defined by
\begin{e*}
    d\langle \cdot \rangle_\alpha (\sigma) \propto \exp \left( \sqrt{2t}H^\alpha_N(\sigma) - Nt \xi_\alpha \left( \frac{\sigma \sigma^\perp}{N}\right)\right) \d P_N(\sigma).
\end{e*}
\begin{proposition} \label{p.lip fe}
    For every $t \geq 0$, there exists a constant $C \geq 0$ such that for every $N \in \N$ and every $\alpha,\alpha' \geq 0$,
    \begin{equation*}
        |\bar F^\alpha_N(t) - \bar F^{\alpha'}_N(t)| \leq  C|\alpha-\alpha'|.
    \end{equation*}
\end{proposition}

\begin{proof} 
    Without loss of generality, we may assume that $t = 1/2$.  The function $\alpha \mapsto \bar F_N^\alpha(1/2)$ is continuous on $\R_+$ and differentiable on $(0,+\infty)$, and we have 
   \begin{e*}
           \frac{\d}{\d\alpha }  F_N^\alpha(1/2) = -\frac{1}{N} \E \left\langle \frac{1}{2\sqrt{\alpha}}H_N^\text{Potts}(\sigma) - \frac{N}{2}\left|\frac{\sigma \sigma^\perp}{N}\right|^2 \right\rangle_\alpha.  
   \end{e*}
   Using the Gaussian integration by part formula \cite[Lemma~1.1]{pan}, it follows that,
   \begin{e*}
        \frac{\d}{\d\alpha }  F_N^\alpha(1/2) = \frac{1}{2} \E \left\langle \left| \frac{\sigma\tau^\perp}{N} \right|^2 \right \rangle_\alpha.
   \end{e*}
   Since the reference measure $P_1$ is compactly supported, there exists a constant $c \geq 0$ such that $\E \left\langle \cdot \right \rangle_\alpha$-almost surely $\left| \frac{\sigma\tau^\perp}{N} \right|^2 \leq D^2c^2$, the result follows.
\end{proof}

For every $\chi \in \mcl L^0_{\leq 1}$, we let
\begin{e*}
    \psi^\alpha_*(\chi) = \inf_{\mu \in \mcl P^\upa_{\xi_\alpha}} \left\{ \int \chi \d\mu - \psi(\mu) \right\},
\end{e*}
and
\begin{e*}
    \tilde S_t^\alpha \chi = \sup_{x \in B(0,1) \cap S^D_+} \left\{ \chi( t\nabla \xi_\alpha(x)) - t\theta_\alpha(x)\right\},
\end{e*}
where,
\begin{e*}
    \theta_\alpha(x) = x \cdot \nabla \xi_\alpha(x) - \xi_\alpha(x) = \theta(x) + \alpha |x|^2.
\end{e*}
\begin{proposition} \label{p.lipschitz hopf}
    For every $t \geq 0$, the function 
    \begin{e*}
        \alpha \mapsto \inf_{\chi \in \mcl L^0_{\leq 1}} \left\{\tilde S_t^\alpha \chi(0) -  \psi^\alpha_*(\chi)\right\},
    \end{e*}
    is Lipschitz on $\R_+$.
\end{proposition}

\begin{proof}
    If for every $\chi \in \mcl L^0_{\leq 1}$, the functions $\alpha \mapsto \tilde S_t^\alpha \chi(0)$ and $\alpha \mapsto  \psi^\alpha_*(\chi)$ are Lipschitz functions on $\R_+$ with Lipschitz constant independent of $\chi$, then the result follows because in this case $\alpha \mapsto \inf_{\chi \in \mcl L^0_{\leq 1}} \left\{\tilde S_t^\alpha \chi(0) -  \psi^\alpha_*(\chi)\right\}$ is the infimum of a family of uniformly Lipschitz functions.

    \noindent Step 1. We show that for every $\chi \in \mcl L^0_{\leq 1}$, $\alpha \mapsto \tilde S_t^\alpha \chi(0)$ is $3t-$Lipschitz.

    \noindent Let $x \in B(0,1) \cap S^D_+$, the function 
    \begin{e*}
        \alpha \mapsto \chi(t\nabla \xi(x) + 2t\alpha x) - t\theta(x) - t\alpha |x|^2
    \end{e*}
    is $3t$-Lipschitz on $\R_+$ and we have 
    \begin{e*}
        \tilde S_t^\alpha \chi = \sup_{x \in B(0,1)} \left\{\chi(t\nabla \xi(x) + 2t\alpha x) - t\theta(x) - t\alpha |x|^2 \right\}.
    \end{e*}
    Therefore, $\alpha \mapsto \tilde S_t^\alpha \chi(0)$ is $3t$-Lipschitz on $\R_+$ as the supremum of a family of $3t$-Lipschitz functions on $\R_+$.
    
    \noindent Step 2.  We show that for every $\chi \in \mcl L^0_{\leq 1}$, $\alpha \mapsto \psi_*^\alpha(\chi)$ is $4-$Lipschitz on $\R_+$.

    \noindent Let $\alpha,\alpha' \geq 0$, we have for every $\mu \in \mcl P^\upa(B(0,1) \cap S^D_+)$
    \begin{align*}
      \left| \int \chi(\nabla \xi(x) + 2 \alpha x) \d\mu(x) -  \int \chi(\nabla \xi(x) + 2 \alpha' x) \d\mu(x) \right| &\leq \int 2|\alpha - \alpha'||x| \d\mu \\
                                                                                                                     &\leq 2|\alpha - \alpha'|.
    \end{align*}
    In addition, since $\psi$ is $1$-Lipschitz with respect to the optimal transport distance $d$, we have 
    \begin{align*}
       &\left| \psi \left( (\nabla \xi +2 \alpha \text{id}_{S^D_+})_\# \mu \right) -  \psi \left( (\nabla \xi +2 \alpha' \text{id}_{S^D_+})_\# \mu \right) \right| \\
       &\leq d( (\nabla \xi +2 \alpha \text{id}_{S^D_+})_\# \mu, (\nabla \xi +2 \alpha' \text{id}_{S^D_+})_\# \mu) \\
       &\leq \sup_{\chi_0 \in \mcl L^0_{\leq 1} } \left\{ \int \chi_0(\nabla \xi(x) +2 \alpha x) \d\mu(x) - \int \chi_0(\nabla \xi(x) +2 \alpha' x) \d\mu(x) \right\} \\
       &\leq \sup_{\chi_0 \in \mcl L^0_{\leq 1} } \left\{ \int 2 |\alpha-\alpha'| |x| \d\mu(x) \right\} \\
       &\leq 2|\alpha-\alpha'|.
    \end{align*}
    Finally, since 
    \begin{e*}
        \psi^\alpha_*(\chi) = \inf_{\mu \in \mcl P^\upa(B(0,1) \cap S^D_+)} \left\{ \int \chi(\nabla \xi(x) +2\alpha x ) \d\mu(x) - \psi \left( (\nabla \xi +2 \alpha \text{id}_{S^D_+})_\# \mu \right) \right\},
    \end{e*}
    we deduce that $\alpha \mapsto \psi^\alpha_*(\chi)$ is $4$-Lipschitz as the infimum of a family of $4$-Lipschitz functions.
\end{proof}

\begin{proof}[Proof of Theorem~\ref{t.main}]
    Let $\alpha > 0$, $\xi_\alpha$ is strictly convex on $S^D_+$ so according to Theorem~\ref{t.main strictly convex}, we have  
    \begin{e*}
        \lim_{N \to +\infty} \bar F^\alpha_N(t) = \inf_{\chi \in \mcl L_{\leq 1}} \left\{ \tilde{S}^\alpha_t \chi(0) - \psi^\alpha_*(\chi) \right\}.
    \end{e*}
    According to Propositions~\ref{p.lip fe}~and~\ref{p.lipschitz hopf}, we can let $\alpha \to 0$ in the previous display to obtain 
    \begin{e*}
        \lim_{N \to +\infty} \bar F_N(t) = \inf_{\chi \in \mcl L_{\leq 1}} \left\{ \tilde{S}_t \chi(0) - \psi^\xi_*(\chi) \right\}.
    \end{e*}
\end{proof}

\section{Interpretations of the main results} \label{s.interpretation}

This last section has a more speculative flavor, its aim is to give an interpretation of the main results of this paper through the lens of Hamilton-Jacobi equations. We start by explaining why we think that \eqref{e.main scalar} can be interpreted as a Hopf-like formula. Assuming that this interpretation can be made rigorous in the case $D > 1$, we construct a conjectural variational formula for the limit free energy when $\xi$ is not assumed to be convex on $S^D_+$.
\subsection{Hopf and Hopf-like formulas} \label{ss.hopf}

Let $\psi : \R^n \to \R$ be a Lipschitz function and let $H : \R^n \to \R$ be a locally Lipschitz function. It is well known that the Hamilton-Jacobi equation 
\begin{e} \label{e.hj interp}
    \begin{cases}
        \partial_t u - H(\nabla u ) = 0 \textrm{ on } (0,+\infty) \times \R^n \\
        u(0,\cdot) = \psi
    \end{cases}
\end{e}
admits a unique viscosity solution $u$ and that when $\psi$ is concave, $u$ admits the Hopf representation \cite[Theorem~3.13]{HJbook}. We recall this in the next proposition.
\begin{proposition}[\cite{HJbook}] \label{p.hopf}
    Assume that $\psi : \R^n \to \R$ is Lipschitz and concave, assume that $H : \R^n \to \R$ is locally Lipschitz. Then, the unique viscosity solution of \eqref{e.hj interp} satisfies 
    \begin{e} \label{e.hopf}
    u(t,x) = \inf_{y \in \R^n} \left\{ x \cdot y - \psi_*(y) +t H(y) \right\},
    \end{e}
    where $\psi_*(y) = \inf_{x \in \R^n} \left\{ x \cdot y - \psi(y) \right\}$. 
\end{proposition}
Intuitively, the Hopf formula follows from the following observations. At time $t = 0$, $u(0,\cdot)$ coincides with the infimum of the family of values taken by affine functions that upper bound $\psi$ (this is Fenchel-Moreau duality),
\begin{e} \label{e.easy fm}
    u(0,x) = \inf_{y \in \R^n} \left\{ x \cdot y - \psi_*(y) \right\}.
\end{e}
For every $y \in \R^n$ the viscosity solution of \eqref{e.hj interp} with the affine initial condition $x \mapsto x \cdot y - \psi_*(y)$ is 
\begin{e*}
    (t,x) \mapsto x \cdot y - \psi_*(y) +tH(y).
\end{e*}
Provided that we can interchange the semigroup of \eqref{e.hj interp} with the infimum in \eqref{e.easy fm}, we obtain \eqref{e.hopf}. 

Notice that in Proposition~\ref{p.hopf}, the geodesics used to define the notion of concavity and the notion of derivative (the symbol $\nabla$ in \eqref{e.hj interp}) are the same, both notions are defined using straight lines in $\R^n$.

For the rest of this subsection we assume that $D = 1$. Let $\psi : \mcl P_1(\R_+) \to \R$ be a nondecreasing Lipschitz function, we have explained that \eqref{e.hj} could be reformulated into \eqref{e.hj frechet}. Now observe that in \eqref{e.hj frechet} the notion of derivative (the symbol $\nabla$ in \eqref{e.hj frechet}) is the Fréchet derivative in $L^2$, it is defined using straight lines in $L^2$. Hence, as might be expected, it was shown in \cite[Theorem~4.6~(3)]{chen2022hamilton} that the Hopf representation is also available for viscosity solutions of \eqref{e.hj}, provided that the function $\msf q \mapsto \psi(\Omega^{-1} \msf q)$ is concave. In this case, $\msf f$ the unique solution of \eqref{e.hj frechet}, satisfies 
\begin{e} \label{e.hopfinf}
    \msf f(t,\msf q ) = \inf_{\msf p \in \mcl Q_\infty(\R_+)} \left\{ \langle\msf q, \msf p \rangle_{L^2} - \psi_{\circ}(\msf p) +t \int \xi(\msf p)  \right\},
\end{e}
where 
\begin{e*}
    \psi_\circ(\msf p) = \inf_{\msf q \in \mcl Q_2(\R_+)} \left\{ \langle \msf q, \msf p \rangle_{L^2} - \psi(\Omega^{-1} \msf q) \right\}.
\end{e*}
Even thought \eqref{e.hopfinf} is the natural generalization of \eqref{e.hopf}, this formula seems to be of limited use in the context of spin glasses since for the function $\psi$ defined by \eqref{e.def_psi}, the function $\msf q \mapsto \psi(\Omega^{-1} q)$ is not concave (nor convex) in general as shown in \cite[Section~6]{mourrat2020nonconvex}.

We now assume that $\psi$ is given by \eqref{e.def_psi}. According to \cite{auffinger2015parisi}, the function $\mu \mapsto \psi(\mu)$ is concave on $\mcl P_2(\R_+)$. When considering \eqref{e.hj} in this case, the geodesics used to define concavity and derivatives are not the same. We use straight lines in $\mcl P_2(\R_+)$ to define concavity and transport geodesics in $\mcl P_2(\R_+)$ (that is straight lines in $L^2$) to define the symbol $\partial_\mu$ appearing in \eqref{e.hj}. Nonetheless, the formula \eqref{e.main scalar} seems to be the natural adaptation of \eqref{e.hopfinf} under this slightly unusual setup. Indeed, let $f$ be the viscosity solution of \eqref{e.hj}, recall that according to Lemma~\ref{l.fenchel moreau}, we have 
\begin{e*}
    f(0,\mu) = \inf_{\chi \in \mcl X} \left\{ \int \chi \d\mu - \psi_*(\chi) \right\},
\end{e*}
where $\mcl X$ denotes the set of Lipschitz functions $\chi : \R_+ \to \R$ which are nondecreasing and convex. Theorem~\ref{t.main scalar} states that 
\begin{e} \label{e.aa}
    f(t,\mu) = \inf_{\chi \in \mcl X} \left\{ \int S_t \chi \d\mu - \psi_*(\chi) \right\},
\end{e}
and according to Theorem~\ref{t.viscosity with linear initial condition}, $(t,\mu) \mapsto \int S_t \chi \d\mu - \psi_*(\chi)$ is the unique viscosity solution of 
\begin{e*}
    \begin{cases}
        \partial_t u - \int \xi(\partial_\mu u) \d\mu  = 0 \textrm{ on } (0,+\infty) \times \mcl P_2(\R_+) \\
        u(0,\mu) = \int \chi \d\mu - \psi_*(\chi).
    \end{cases}
\end{e*}
Therefore \eqref{e.aa} and \eqref{e.hopfinf} have the same structure; those two formulas express the viscosity solution as the infimum of a family of viscosity solutions started at an affine initial condition. The only difference is that for \eqref{e.hopfinf} the relevant linear initial conditions are of the form $\msf q \mapsto \langle \msf p , \msf q \rangle_{L^2}$ when for \eqref{e.aa} the relevant linear initial conditions are of the form $\mu \mapsto \int \chi \d\mu$. This is why we refer to \eqref{e.main scalar} as a Hopf-like formula.

\subsection{Some conjectures}

In this subsection, we do not assume that $D = 1$. When $\xi$ is not assumed to be convex on $S^D_+$, the Parisi formula breaks down and there is no known generalization of Theorem~\ref{t.parisi}. In \cite{mourrat2020nonconvex, mourrat2019parisi, mourrat2020free} it was conjectured that results such as Theorem~\ref{t.limit.free.energy} should generalize to the setting where the function $\xi$ is not assumed to be convex on $S^D_+$. Namely, it should hold under \ref{i.assume_1_product_proba} and \ref{i.assume_3_power_series} only that the free energy converges as $N \to +\infty$ to the viscosity solution of \eqref{e.hj}. When $\xi$ is not assumed to be convex the Hopf-Lax representation for the viscosity solutions of \eqref{e.hj} is not available and variational representations for the limit free energy such as \eqref{e.lim.FN} are proven to be false (see \cite[Section~6]{mourrat2020nonconvex}). 

As explained in the previous subsection, for Hamilton-Jacobi equations with possibly nonconvex nonlinearity, if the initial condition is concave, a variational representation is available for the viscosity solution, and it seems that the un-inverted formula \eqref{e.main scalar} can be interpreted as an instance of such a variational representation. One of the main ingredients for proving \eqref{e.main scalar} and coincidently also for the Hopf representation to hold is Fenchel-Moreau duality (which manifests through Lemma~\ref{l.fenchel moreau} here). If one wishes to generalize \eqref{e.main scalar} to models with $D > 1$, a seemingly crucial step is thus to generalize Lemma~\ref{l.fenchel moreau}. The attentive reader will notice that a version of Lemma~\ref{l.fenchel moreau} follows from Theorem~\ref{t.extension}, this allows us to write $\psi$ as an infimum over Lipschitz functions $\chi : S^D_+ \to \R$. But this version of Lemma~\ref{l.fenchel moreau} is not strong enough for our purpose. Indeed, to guarantee that \eqref{e.small hj diag} below is well-posed, one needs that the initial condition $\chi$ is nondecreasing \cite[Theorem~1.2~(2)]{chen2023viscosity}. The argument we have used in the proof of Lemma~\ref{l.fenchel moreau} to show that the infimum could in fact be taken over nondecreasing $\chi$'s doesn't seem to be easily applicable when $D > 1$. Indeed, we have used crucially the fact that the set of surjective paths $\msf q \in \mcl Q(\R_+)$ is dense in $\mcl Q(\R_+)$. When $D > 1$, the set of surjective paths $\msf q \in \mcl Q(S^D_+)$ is empty, as any surjective function $[0,1) \to S^D_+$ must be non-monotonous. In order to obtain a suitable generalization of Lemma~\ref{l.fenchel moreau} to the case $D > 1$, we thus make the following additional assumptions.
\begin{enumerate}[start=4,label={\rm (H\arabic*)}]
    \item \label{i.assume_4_product_proba}
    There exists compactly supported probability measures $\pi_1, \dots, \pi_D \in \mcl P_\infty(\R)$ such that $P_1 = \pi_1 \otimes \dots \otimes \pi_D$.
    \item \label{i.assume_5_diag}
    The function $\xi$ only depends on the diagonal coefficients of its argument. That is, there exists a function $\bar \xi : \R^D \to \R$ such that $\xi(A) = \bar \xi((A_{dd})_{1 \leq d \leq D})$.
\end{enumerate}
Thanks to \ref{i.assume_5_diag} we can encode the limit free energy with a partial differential equation on $\mcl P^\upa_2(\R^D_+)$ rather than $\mcl P^\upa_2(S^D_+)$, and using \ref{i.assume_4_product_proba} we can adapt Lemma~\ref{l.fenchel moreau}. Let us explain this in more details.

The map $x \mapsto \textrm{diag}(x)$ which maps each vector $x \in \R^D_+$ to the matrix in $S^D_+$ whose diagonal coefficients are $x_1,\dots,x_D$, defines an injection from $\R^D_+$ to $S^D_+$. In particular each $\mu \in \mcl P^\upa_1(\R^D_+)$ can be interpreted as a probability measure in $\mcl P^\upa_1(S^D_+)$. This means that the quantity $\bar F_N(t,\mu)$ is also defined when $\mu \in \mcl P_1(\R^D_+)$. We can then easily adapt the arguments of \cite[Section~8]{chenmourrat2023cavity} to prove the following theorem.
\begin{theorem}[limit free energy for convex diagonal models] \label{t.limit.free.energy.diag}
    Assume that \ref{i.assume_1_product_proba}, \ref{i.assume_2_convexity}, \ref{i.assume_3_power_series} and \ref{i.assume_5_diag} hold, then for every $t \ge 0$ and $\mu \in \mcl P^\upa_1(\R^D_+)$, we have
    \begin{equation}  \label{e.lim.FN.diag}
        \lim_{N \to +\infty} \bar F_N(t,\mu) = \sup_{\nu \in \mcl P^\upa_\infty(\R^D_+), \nu \ge \mu} \Ll( \psi(\nu) - t \E \Ll[ \bar \xi^* \Ll( \frac{X_\nu - X_\mu}{t} \Rr)  \Rr]  \Rr) .
    \end{equation}
    Moreover, denoting by $f(t,\mu)$ the expression above, we have that $f : \R_+ \times \mcl P^\upa_2(\R^D_+) \to \R$ solves the Hamilton-Jacobi equation
    \begin{equation} \label{e.hj.diag}
        \Ll\{
        \begin{aligned}  %\label{}
        & \dr_t f - \int \bar \xi(\dr_\mu f) \, \d \mu = 0 \quad & \text{ on } \R_+ \times \mcl P^\upa_2(\R^D_+), \\
        & f(0,\cdot) = \psi & \text{ on } \mcl P^\upa_2(\R^D_+). 
        \end{aligned}
        \Rr.
    \end{equation}
\end{theorem}
Let $\psi$ be the functional defined in \eqref{e.def_psi} and let $\psi_d : \mcl P_1(\R_+) \to \R$ denote the function obtained from \eqref{e.def_psi} when $D = 1$ and the reference probability measure is $\pi_d$. Given $\mu \in \mcl P(\R^D_+)$ let $\mu_1, \dots,\mu_D \in \mcl P(\R_+)$ denote its marginals. Assumption \ref{i.assume_4_product_proba} yields the following simplification, for every $\mu \in \mcl P(\R^D_+)$,
\begin{e} \label{e.decouple}
    \psi(\mu) = \sum_{d = 1}^D \psi_d(\mu_d).
\end{e}
Let $\mcl X_D$ denote the set of functions $\chi : \R_+^D \to \R$ which satisfy 
\begin{e} \label{e.def mcl X_D}
    \chi(x) = \sum_{d = 1}^D \chi_d(x_d),
\end{e}
with $\chi_1,\dots,\chi_D : \R_+ \to \R$ $1$-Lipschitz, nondecreasing and convex. The set $\mcl X_D$ is contained in the set of Lipschitz, nondecreasing and convex functions on $\R^D_+$. In particular if we combine \eqref{e.decouple} with Lemma~\ref{l.fenchel moreau} applied to each $\psi_d$, we obtain the following for every $\mu \in \mcl P_1^\upa(\R^D_+)$,
\begin{e} \label{e.dual vector}
    \psi(\mu) = \inf_{\chi \in \mcl X_D} \left\{ \int \chi \d\mu - \psi_*(\chi) \right\},
\end{e}
where, as usual, 
\begin{e}
    \psi_*(\chi) = \inf_{\mu \in \mcl P_1^\upa(\R^D_+)} \left\{ \int \chi \d\mu - \psi(\mu) \right\}.
\end{e}
With \eqref{e.dual vector} in mind, we can expect that the mechanism leading to \eqref{e.aa} in the case $D = 1$, leads to the following result when $D \geq 1$.
\begin{conjecture} \label{c.hopf}
    Assume that \ref{i.assume_1_product_proba}, \ref{i.assume_3_power_series}, \ref{i.assume_4_product_proba} and \ref{i.assume_5_diag} hold but with $\bar \xi$ possibly non-convex on $\R^D_+$, then $f$ the unique viscosity solution of \eqref{e.hj.diag} satisfies 
    \begin{e*}
        f(t,\mu) = \inf_{\chi \in \mcl X_D} \left\{ \int S_t\chi \d\mu - \psi_*(\chi) \right\},
    \end{e*}
    where $S_t \chi(x) = \sup_{y \in \R^D_+} \left\{ x \cdot y - \chi^*(y) + t\bar \xi(y) \right\}$ is the unique viscosity solution of 
    \begin{e*}
        \begin{cases}
            \partial_t u - \bar \xi(\nabla u) = 0 \textrm{ on } (0,+\infty) \times \R^D_+ \\
            u(0,\cdot) = \chi.
        \end{cases}
    \end{e*}
\end{conjecture}
Combining Conjecture~\ref{c.hopf} with \cite[Conjecture~2.6]{mourrat2019parisi}, we obtain the following conjectural variational formula for the limit free energy of nonconvex models.
\begin{conjecture} \label{c.var formula}
    Assume that \ref{i.assume_1_product_proba}, \ref{i.assume_3_power_series}, \ref{i.assume_4_product_proba} and \ref{i.assume_5_diag} hold but with $\bar \xi$ possibly non-convex on $\R^D_+$, then
    \begin{e*}
        \lim_{N \to +\infty } \bar F_N(t,\delta_0) = \inf_{\chi \in \mcl X_D} \left\{ S_t\chi(0) - \psi_*(\chi) \right\},
    \end{e*}
    where $\mcl X_D$ is defined in \eqref{e.def mcl X_D}, $S_t \chi(0) = \sup_{y \in \R^D_+} \left\{- \chi^*(y) + t \bar \xi(y) \right\}$ is the value at $(t,0)$ of the unique viscosity solution of 
    \begin{e} \label{e.small hj diag}
        \begin{cases}
            \partial_t u - \bar \xi(\nabla u) = 0 \textrm{ on } (0,+\infty) \times \R^D_+ \\
            u(0,\cdot) = \chi,
        \end{cases}
    \end{e}
    and where 
    \begin{e}
        \psi_*(\chi) = \inf_{\mu \in \mcl P_1^\upa(\R^D_+)} \left\{ \int \chi \d\mu - \psi(\mu) \right\}.
    \end{e}
\end{conjecture}
\newpage 

\small
\bibliographystyle{plain}
\bibliography{uninverting}

\end{document}